\documentclass[oneeqnum,final]{siamltex}

\usepackage[latin1]{inputenc}
\usepackage[T1]{fontenc}
\usepackage{amssymb,amsmath,graphicx}
\usepackage[mathscr]{eucal}
\usepackage{outils}

\graphicspath{{../}{./}}

\author{ 
  Alex Bombrun\thanks{Most of this work was done when this author was with INRIA. \texttt{Alex.Bombrun@gmail.com}.}
  \and 
  Jean-Baptiste Pomet\thanks{INRIA, B.P. 93, 06902 Sophia Antipolis cedex, France. 
    Email:~\texttt{Jean-Baptiste.Pomet@inria.fr}.} 
}

\title{The averaged control system \\ of fast oscillating control systems\thanks{
This work was partly supported by
    \textit{Thales Alenia Space}, in 2004-2007.\newline
\textsf{Submitted to SIAM J. Control. Optim., Dec. 5, 2011. Revised Sept. 10 and Dec. 7, 2012.} }} 

\begin{document}

\maketitle

\begin{abstract} 
For control systems that either have a fast explicit periodic dependence on time and bounded controls or have periodic solutions and small
controls, we define an \emph{average control system} that takes into account all possible variations of the control, and
prove that its solutions approximate all solutions of the oscillating system as oscillations go faster.

The dimension of its velocity set is characterized geometrically. When it is maximum the average system defines a Finsler
metric, not twice differentiable in general. 
For minimum time control, this average system allows one to give a rigorous proof that averaging the Hamiltonian given by the maximum principle is a valid approximation. 
\end{abstract}



\pagestyle{myheadings}
\thispagestyle{plain}
\markboth{A. BOMBRUN AND J.-B. POMET}{AVERAGED CONTROL SYSTEM}

\section{Introduction}
\label{sec-intro}
We consider either a ``fast-oscillating control system'' \eqref{eq-sysO}:
\begin{equation*}
\dot x = u_1 X_1(\frac t\varepsilon,x)+\cdots+u_m X_m(\frac t\varepsilon,x)\,,\ \|u\|\leq1\,, 
\end{equation*}
where all $X_i$'s are $2\pi$-periodic with respect to $t/\varepsilon$, or a ``Kepler control system'' \eqref{eq:Kep}:
$$\dot\xi = f_0(\xi)+v_1 f_1(\xi)+\cdots+v_m f_m(\xi) \,,\ \|v\|\leq\varepsilon$$
where all solutions of $\dot\xi = f_0(\xi) $ are periodic.

Averaging techniques 
for
conservative ---periodic or not--- ordinary differential equations (ODEs)
date back at least to H. Poincaré; see \cite[\S52]{Arno89} or \cite{Sand-Ver85} for recent
expositions.
Roughly speaking,  
on a fixed interval, the solutions of $\dot x=F(t/\varepsilon,x)$ differ from those of 
$\dot x=\overline{F}(x)$ by a term of order $\varepsilon$, with $\overline{F}$ the average of $F$ with respect to its first argument. 

If $u$ or $v$ above is assigned to be a fixed function of state and time (or computed from
additional state variables as in $u=\alpha(p,x),\;\dot{p}=g(p,x)$), then these techniques for ODEs can be applied to give an
approximation at first order with respect to small $\varepsilon$ of the movement of the slow variables. 
Averaging is usually used in this way in control theory:
in vibrational control~\cite{Meer80}, fast oscillating controls are designed and 
averaging techniques allows analysis and proof of stability;
in the same way, it solves stability and path planning questions in control of mechanical systems, see for instance \cite{Bull02};
in \cite[\S5]{Flie-Lev-Mar-R95ijc}, high frequency control is used to approach a non-flat system by a flat one;
one may also mention many applications to control~\cite{Murr94,Liu97siam2,Mori-Pom-Sam99} of the 
work~\cite{Kurz-Jar88} that mimics Lie brackets by highly oscillatory controls along the original vector fields.
A common feature to these references is that the use of oscillations ``creates'' new independent controls used for design. 
The use of averaging in optimal control of oscillating systems~\cite{Chap87,Geff97th,Gef-Epe:97,Bon-Cail:06:AIHP} is
similar in spirit to the above, but closer to the framework of this paper because oscillations are present in the system
instead of being introduced by the control.
Very interesting results are obtained applying averaging to the Hamiltonian
equations arising from Pontryagin Maximum principle.
For instance, in~\cite{Bon-Cail:06:AIHP}, the authors have studied in this way the problem of minimal energy transfer
between two elliptic orbits; extremals are the same as
those giving the geodesics of a Riemannian metric. 
Again, averaging introduces ``new independent controls'': Riemannian geodesics are minimizers of a problem where all
velocity directions are allowed whereas the velocity set of the
original system at each point had positive codimension.
The same averaging computation may be applied to the Hamiltonian differential equation obtained for minimum time, but,
since this differential equation is discontinuous, there is no theoretical justification for averaging in that case.

Our contribution is to introduce a different way of averaging that takes into account all possible variations of the
control ---hence the control strategy can be decided \emph{after} performing averaging--- and to prove that it has
satisfying regularity properties and is a good first order approximation of the above systems as $\varepsilon\!\to\!0$.
This gives, as a side result, a justification of the use of averaging for minimum time
in~\cite{Geff97th,Gef-Epe:97}.
This procedure also ``creates new independent control'', i.e. increases the dimension of the velocity set, that we
characterize in terms of the original vector fields. When this dimension is maximum, the average system
defines a Finsler metric~\cite{Bao-Che-She00} on the manifold, whose geodesics are the limits of minimum time trajectories for the original
systems as $\varepsilon\!\to\!0$. 
This Finsler metric is in general not twice differentiable (hence it is not a Finsler metric in the sense of
~\cite{Bao-Che-She00}, indeed); 
we however prove that, at least in the less degenerate case, the Hamiltonian system governing extremals, although it is
not locally Lipschitz, generates a flow on the cotangent bundle.
Low thrust planar orbit transfer belongs to this less degenerate case.

The average control system may be used for other purposes than optimal control, for instance~\cite{Bomb07th} designs a
Lyapunov function for feedback control in the average system and uses it for the oscillating systems;
indeed the present work was developed out of comparing feedback control based on a priori chosen Lyapunov functions
with minimum time control for low thrust orbital transfer.

Preliminary versions of this paper can be found in~\cite{bom-pom:06:MTNS,Bomb07th}. It is organized as follows:
the construction and results are developed for ``fast-oscillating control system'' in \S\ref{sec-fast} and then transferred in \S\ref{sec-kep} to ``Kepler
control systems'', and applied to minimum time orbit transfer in the planar 2-body problem in  \S\ref{sec-2body}.

\section{Notations and conventions}
$\ $

\refstepcounter{subsection}
\textbf{\thesubsection.} $\vari$ is a smooth connected manifold of dimension $n$; its tangent
and cotangent bundles are denoted by $\mathrm{T}\vari$ and $\mathrm{T}^*\vari$.
One may assume for simplicity $\vari=\RR^n$, $\mathrm{T}\vari=\RR^n\times\RR^n$,
$\mathrm{T}^*\vari=\RR^n\times\left(\RR^n\right)^*$, and, for $x\in\vari$, $\mathrm{T}_x\vari=\RR^n$,
$\mathrm{T}^*_x\vari=\left(\RR^n\right)^*$. 

For $v\in\mathrm{T}_x\vari$, $p\in\mathrm{T}^*_x\vari$ (or any $v,p$ taken in a vector space and its dual), we denote by
$\langle p,v\rangle$ (rather than $p(v)$) their duality product. 

\refstepcounter{subsection}
\textbf{\thesubsection.} 
If $E$ is a subset of a vector space $V$, then $E^\perp$ is its annihilator, the vector subspace of its dual $V^*$ 
made of all $p$'s such that $\langle p,v\rangle=0$ for all $v$ in $E$.

\refstepcounter{subsection}\label{sec-not-d}
\textbf{\thesubsection.} 
We assume that $\vari$ is endowed with an arbitrary Riemannian distance $d$. 
If $\vari=\RR^n$, just choose the canonical Euclidean distance.

In local coordinates, $\|.\|$ and $\scal{.}{.}$ stand for the canonical Euclidean norm and scalar product.
On a compact coordinate chart, $k_1 \|x-y\|\leq d(x,y)\leq k_2 \|x-y\|$ for some positive $k_1$, $k_2$
(Lipschitz equivalence).
We also denote
operator norms by $\|.\|$.

\refstepcounter{subsection}\label{sec-not-S1}
\textbf{\thesubsection.} 
$S^1$ is $\RR/2\pi\ZZ$. For $\theta$ in $S^1$ (an angle), we denote by $\mu(\theta)$ the unique real number in $[0,2\pi)$ such that
$\mu(\theta)\equiv\theta\mod2\pi$. For a real number $s\in\RR$, we denote the angle it represents by
$s\!\!\mod\!2\pi$; it belongs to the quotient $S^1$.

Maps $S^1\to E$ (arbitrary set) are identified with $2\pi$-periodic maps $\RR\to E$.
For instance, if $f$ is such a map $S^1\to E$ and $\tau\in\RR$, we write $f(\tau)$ instead of
$f(\tau\!\!\!\mod\!2\pi)$;
the average of $f$ is denoted by 
$\frac1{2\pi} \! \int_0^{2\pi} \!\!\!f(\theta)\mathrm{d}\theta$, 
or $\frac1{2\pi} \! \int_{-\pi}^{\pi} \!f(\theta)\mathrm{d}\theta$, 
or $\frac1{2\pi}\! \int_{\theta\in S^1} f(\theta)\mathrm{d}\theta$; one identifies
$L^p(S^1,\RR^m)$ with the subset of $L^p(\RR,\RR^m)$ made of $2\pi$-periodic functions.

\refstepcounter{subsection}\label{sec-not-Rm}
\textbf{\thesubsection.} 
The Euclidean norm in $\RR^m$ or $(\RR^m)^*$ is denoted by $\|.\|$, and the ball of
radius one centered at the origin by $B^m$. 
We view an element of $\RR^m$ as $m\times1$ matrix (column) of real numbers and an element of $(\RR^m)^*$ as a $1\times
m$ matrix (line); transposition, denoted
$.^\top$, sends $\RR^m$ to $(\RR^m)^*$ and vice-versa.


%
%
\section{Fast oscillating control systems}
\label{sec-fast}

We call \emph{fast oscillating control system} on $\vari$ 
a family of non-autonomous systems, linear in the control $u\in\RR^m$:
\begin{equation}
\label{eq-sysO}
\dot{x} = \GG(\frac{t}\varepsilon,x)\,u = \sum_{i=1}^m \GG_i(\frac{t}\varepsilon,x)\,u_i \,, \|u\|\leq 1\,
\end{equation}
indexed by a positive number $\varepsilon$. 
Each $\GG_i$ is a smooth
``periodic time-varying'' vector field:
$\GG_i\in C^{\infty}(S^1\times\vari, \mathrm{T}\vari)$.
An admissible control is a measurable $u(.):[0,T]\to B^m$ for some $T>0$. For a given control $u(.)$ and initial condition
$x(0)$, there is a unique solution $x(.)$, defined either on $[0,T]$ or only on 
a maximal interval $[0,T')$, $T'<T$.

\refstepcounter{theorem}\paragraph{Remark~\thetheorem}\label{rmk-G-oper}
Apart from being a notation defined by the double
equality in (\ref{eq-sysO}), $\GG(\theta,x)$ defines a linear map $\RR^m\to\mathrm{T}_x\vari$ that sends
$(u_1,\ldots,u_m)^\top$ to $\sum_{i=1}^m \GG_i(\frac{t}\varepsilon,x)u_i$.

\subsection{Average control system of fast oscillating control systems}
\label{sec-fast-aver}

Define the map $\overline{\GG}:\vari\times L^\infty([0,2\pi],\RR^m)\to\mathrm{T}\vari$ by
\begin{equation}
  \label{eq:defGbar}
  \overline{\GG}(x,\mathscr{U})=\frac1{2\pi}\int_0^{2\pi}\!\!\GG(\theta,x)\,\mathscr{U}(\theta)\,\mathrm{d}\theta\ .
\end{equation}
It allows one to define, for all $x\in\vari$, the subset $\EE(x)\subset\mathrm{T}_x\vari$ by
\begin{equation}
  \label{eq:calE}
  \EE(x)=\left\{ \overline{\GG}(x,\mathscr{U}), \,\mathscr{U}\in L^\infty\left([0,2\pi], \RR^m
    \right),\,\|\mathscr{U}\|_{\infty}\leq1  \right\}\subset T_x\vari \ ,
\end{equation}
and the \emph{average control system} of (\ref{eq-sysO}) as follows\footnote{Its relation to the limit case  of \eqref{eq-sysO} as
  $\varepsilon\to0$ is discussed in the next section.}.

\smallskip

\begin{definition}
  \label{def-aver}
The average control system of \eqref{eq-sysO} is the differential inclusion
\begin{equation}
\label{eq-sysM}
\dot{x}\in \EE(x).
\end{equation}
A solution of (\ref{eq-sysM}) is an absolutely continuous $x(.):[0,T]\to\vari$ such that 
$\dot x(t)\in\EE(x(t))$ for almost all $t$.
\end{definition}

\smallskip

\begin{proposition}
\label{lem-conv}
For all $x$ in $\vari$, $\EE(x)$ is convex, compact and symmetric with respect to the origin.
\end{proposition}
\begin{proof}
It is closed, convex and symmetric because it is the image of the unit ball of $L^\infty\left(S^1, \RR^m \right)$ by a linear map; it is compact because $\GG(x,.)$ is bounded on $S^1$.
\hfill\end{proof}

\smallskip

Further characterizations of $\EE(x)$ use the map $H:\mathrm{T}^*\vari\to[0,+\infty)$ defined by
\begin{equation}
\label{eq:Hamiltonian}
H(x,p) = \frac{1}{2\pi}\int_0^{2\pi} \| \langle p , \GG(\theta,x) \rangle \| \mathrm{d}\theta
\end{equation}
where the Euclidean norm is used according to \S\ref{sec-not-Rm} and, for each $(\theta,x)$,
\begin{equation}
  \label{eq:12}
  \langle p, \GG(\theta,x) \rangle=\left(\langle p,\GG_1(\theta,x)\rangle,\ldots,\langle p,\GG_m(\theta,x)\rangle\right)\in (\RR^m)^*\,.
\end{equation}
\begin{proposition}
\label{lem-E-H}
For all $(x,p)\in T^*\vari$, one has, with $H$ defined in \eqref{eq:Hamiltonian},
  \begin{gather}
   \label{eq:EH}
   \EE(x)=
\Bigl\{v\in\mathrm{T}_x\vari\,,\ \sup_{\substack{p\in\mathrm{T}^*_x\vari\\H(x,p)\leq1}}\!\!\langle p,v
\rangle\;\leq\;1\Bigr\}
\,,
\\
    \label{eq:EHbis}
    H(x,p)\ =\sup_{v\in\EE(x)}\langle p,v \rangle\ =
\sup_{ \mathscr{U}\in L^\infty\left(S^1, \RR^m \right),\;\|\mathscr{U}\|_{\infty}\leq1}
\langle p,\overline{\GG}(x,\mathscr{U}) \rangle
\ =\ \langle p,\overline{\GG}(x,\mathscr{U}^*_{p,x}) \rangle
\,,
\\
\!\!\!\!\text{with $\mathscr{U}^*_{p,x}\!\in\! L^\infty\left(S^1, \RR^m \right)$ defined by: }
\ 
\mathscr{U}^*_{p,x}(\theta)=
\begin{cases}
    0&\text{if }\langle p, \GG(\theta,x) \rangle=0\,,
\\
  \label{eq:Ustarpx}
    \frac{\langle p, \GG(\theta,x) \rangle^\top}{\|\langle p, \GG(\theta,x) \rangle\|}
    &\text{if }\langle p, \GG(\theta,x) \rangle\neq0\,.
\end{cases}
\end{gather}
\end{proposition}
\begin{proof}
  The last equality in (\ref{eq:EHbis}) is a straightforward maximization, the second one comes from the
definition (\ref{eq:calE}) of $\EE(x)$ and a simple computation yields
$H(x,p)=\langle p,\overline{\GG}(x,\mathscr{U}^*_{p,x}) \rangle$; this proves
(\ref{eq:EHbis}).
Being closed and convex, $\EE(x)$ is the intersection of all its supporting half-spaces~\cite[Corollary 1.3.5]{Schn93}; according to (\ref{eq:EHbis}), this yields the following relation, equivalent to (\ref{eq:EH}):
$\EE(x)=  \bigcap_{p\in\mathrm{T}^*_x\vari}\left\{v\in\mathrm{T}_x\vari\,,\;\langle p,v \rangle\leq H(x,p)\right\}$.\hfill
\hfill\end{proof}

\medskip

\paragraph{A convenient characterization of solutions of \eqref{eq-sysM}}
According to  Definition~\ref{def-aver}, a solution $x(.)$ is such that, for almost all $t$, there is 
$\mathscr{U}(t)\in L^\infty([0,2\pi],\RR^m)$ such that
$\dot{x}(t)=\overline{\GG}(x(t),\mathscr{U}(t))$; the map $(t,\theta)\mapsto\mathscr{U}(t)(\theta)$ is measurable with
respect to $\theta$ only. 
It turns out that it may always be chosen jointly measurable with respect to $(t,\theta)$ according to the following
``measurable selection'' result:

\begin{proposition}
\label{lem-meas}
 A map $\III:[0,T]\to \RR^n$ is a solution of the differential inclusion~\eqref{eq-sysM} if and only if there
 exists $\widehat{u}\in L^\infty([0,T]\times S^1,\RR^m)$, $\|\widehat{u}\|_\infty\leq1$ such that
\begin{equation}
   \label{eq:038}
   \dot{\III}(t)=\frac1{2\pi}\int_0^{2\pi}
 \GG(\theta,\III(t))\,\widehat{u}(t,\theta)\,
 \mathrm{d}\theta 
 \end{equation}
for almost all $t$ in $[0,T]$.
\end{proposition}

\begin{proof}
After possibly partitioning $[0,T]$ into intervals where $\dot{\III}(t)$ remains in the same coordinate
chart, we work in coordinates and use a Euclidean norm when useful.

Sufficiency is clear:
from Fubini theorem, $\theta\mapsto\widehat{u}(t,\theta)$ is measurable for almost all $t$, hence
$\III(.)$ is a solution of (\ref{eq-sysM}).
Conversely, let $\III(.)$ be a solution of \eqref{eq-sysM}: $\dot{\III}(.)$ is measurable and, for almost all $t$, there exists $\tilde{u}_t\in L^\infty(S^1,\RR^m)$, $\|\tilde{u}_t\|_\infty\leq1$ such that
\begin{equation}
  \label{eq:005}
  \dot{\III}(t)=\overline{\GG}(\III(t),\tilde{u}_t)=\frac1{2\pi}\int_0^{2\pi}\GG(s_1,\III(t))\tilde{u}_t(s_1)\mathrm{d}s_1\,.
\end{equation}
Let $\phi:L^\infty\left([0,T]\times S^1,\RR^m\right)\to L^2\left([0,T],\RR^n\right)$ be the linear map defined by 
 $$
 \phi(u)(t)=\overline{\GG}(x(t),u(t,.))=\frac{1}{2\pi}\int_0^{2\pi} \GG(s_1,\III(t))\,u(t,s_1)\, \mathrm{d}s_1
 $$
and $\mathscr{I}$ the image by $\phi$ of the unit ball of $L^{\infty}([0,T]\!\times\! S^1,\RR^m)$. 
Since,  by (\ref{eq:005}), $\dot{\III}(.)$ is essentially bounded, it is in $L^2\left([0,T],\RR^n\right)$; since 
$\mathscr{I}$ is closed and convex in that Hilbert space,  the distance from $\dot\III$ to $\mathscr{I}$
is reached for a unique element $\bar\xi\in\mathscr{I}$:
\begin{displaymath}
  \bar\xi=\phi(\bar{u})\,,\ \ \bar u\in L^\infty\left([0,T]\times S^1,\RR^m\right)\,,\ \ \|\bar u\|_{L^\infty}\leq1\;.
\end{displaymath}
Let us prove by contradiction that $\bar\xi=\dot{\III}$, \textit{i.e.} $\dot\III(.)\in \mathscr{I}$; this will end the
proof.

If $\dot{\III}\neq\bar\xi$, one has, for all $u$ in the unit ball of $L^\infty\left([0,T]\times S^1,\RR^m\right)$,
\begin{equation}
  \label{eq:007}
  \left(\left.\dot\III-\bar\xi\,\right|\phi(u)-\phi(\bar{u})\right)_{L^2}\ \leq\ 0
\end{equation}
with equality only if $\phi(u)=\phi(\bar{u})$.
Define $\widehat{u}$ by $\displaystyle
\widehat{u}(t,s)=
\mathscr{U}^*_{\left(\dot\III(t)-\bar\xi(t)\right)^\top,\,x(t)}(s)
$ with $\mathscr{U}^*_{p,x}$ defined by (\ref{eq:Ustarpx}); clearly, $\widehat{u}$ is in the unit ball of
$L^\infty\left([0,T]\times S^1,\RR^m\right)$, and, for all $(t,s)\in[0,T]\times S^1$ and all $\mathbf{u}\in\RR^m$,
\begin{equation}
\label{eq:004}
  \|\mathbf{u}\|\leq1\ \;\Rightarrow\;\ 
  \left(\dot\III(t)-\bar\xi(t)\right)^\top\GG(s,\III(t))\,\left(\widehat{u}(t,s)-\mathbf{u}\right)\geq0\;,
\end{equation}
hence $\left(\dot\III(t)-\bar\xi(t)\right)^{\!\top}\!\GG(s_1,\III(t))\,\left(\widehat{u}(t,s_1)-\bar{u}(t,s_1)\right)$ is
non-negative for almost all $(t,s_1)$ and, since it is the integrand of the left-hand side of (\ref{eq:007}), it must be
zero; hence $\bar\xi=\phi(\bar{u})=\phi(\widehat{u})$ and
$\bar\xi(t)=\overline{\GG}(\III(t),\widehat{u}(t,.))$ for almost all $t$.
\\In (\ref{eq:005}), $\tilde{u}_t$ satisfies $\|\tilde{u}_t(s_1)\|\leq1$ for almost all $s_1$, hence,
according to (\ref{eq:004}),
$$\left(\dot\III(t)-\bar\xi(t)\right)^\top\GG(s_1,\III(t))\,\left(\widehat{u}(t,s_1)-\tilde{u}_t(s_1)\right)\geq0\;.
$$
Since $\dot{\III}(t)=\overline{\GG}(\III(t),\tilde{u}_t)$, $\bar\xi(t)=\overline{\GG}(\III(t),\widehat{u}(t,.))$,
the integration with respect to the variable $s_1$ yields
$-\|\dot\III(t)-\bar\xi(t)\|^2\geq0$ for almost all $t$; this contradicts $\dot{\III}\neq\bar\xi$.
\hfill\end{proof}

\smallskip

\refstepcounter{theorem}\paragraph{Remark~\thetheorem}\label{rmk-controlDimInf}
The differential inclusion (\ref{eq-sysM}) is equivalent to the ``control system''
\\[1.1ex]
\hspace*{1ex}\hfill
$\dot{x}=\overline{\GG}(x,\mathscr{U})\,,\ \ \ \mathscr{U}\in L^\infty(S^1,\RR^m)\,,\
\|\mathscr{U}\|_{\infty}\leq1 $
\hfill\hspace*{1ex}
\\[1.1ex]
where, 
by Proposition~\ref{lem-meas}, admissible controls 
are maps $t\mapsto\mathscr{U}(t)$ such
that $\widehat{u}\!: (t,\theta)\mapsto\mathscr{U}(t)(\theta)$ is measurable
with respect to $(t,\theta)$.
Since this ``control'' is infinite dimensional, and we could not find a representation 
of the type 
$\dot x=f(x,v),\,v\in U\subset\RR^r$, $r$ finite, we stay with the differential inclusion 
(\ref{eq-sysM}), with $\EE(x)$ described by \eqref{eq:EHbis} and \eqref{eq:Hamiltonian}.



\subsection{Convergence theorem}
\label{sec-fast-conv}
The following result relates solutions of the fast oscillating systems as $\varepsilon$ tends to zero to
solutions of the average system.
To our knowledge, this kind of theorem where the control is not chosen prior to averaging has never been stated in the
literature.

\begin{theorem}[Convergence for fast-oscillating control systems]
  \label{th-conv-sysM}
  \begin{remunerate}
  \item\label{it-thconv-a} 
    Let $\III_0(.):[0,T]\to\vari$ be an arbitrary solution of \eqref{eq-sysM}.
    There exist a family of measurable functions $\overline{u}_\varepsilon(.):[0,T]\to B^m$, indexed by $\varepsilon>0$, and positive constants $c,\varepsilon_0$,
    such that, calling $\III_{\varepsilon}(.)$ the solution of \eqref{eq-sysO} with control
    $u=\overline{u}_\varepsilon(t)$ and initial condition $\III_{\varepsilon}(0)=\III_0(0)$, one has:
\quad
    $\III_{\varepsilon}(.)$ is defined on $[0,T]$ for all $\varepsilon$ smaller than $\varepsilon_0$ and converges to
    $\III_0(.)$ as $\varepsilon\to0$, with an error of uniform order $\varepsilon$:
     \begin{equation}
       \label{eq:CVosc}
       \mathit{d}(\III_{\varepsilon}(t),\III_0(t))<c\,\varepsilon\,,\ \ t\in[0,T]\,,\;0<\varepsilon<\varepsilon_0\,.
     \end{equation}

\item\label{it-thconv-b} Let $\Set{K}$ be a compact subset of $\vari$, 
$(\varepsilon_n)_{n\in\NN}$ a decreasing sequence of positive real numbers converging to zero,
and, for each $n$, $\III_{n}(.)\!:[0,T]\to\Set{K}$ a 
solution of~(\ref{eq-sysO}) with $\varepsilon=\varepsilon_n$ and for some control $u=u_n(t)$, $u_n(.)\in L^\infty([0,T],\RR^m)$, $\|u_n(.)\|_\infty\leq1$.
Then the sequence $\bigl(\III_{n}(.)\bigr)_{n\in\NN}$ is compact for the topology of uniform convergence on
$[0,T]$ and any accumulation point is a solution of the average system~\eqref{eq-sysM}.

 \end{remunerate}
\end{theorem} 

\smallskip

The statement is more complex than the
one for ODEs, \textit{e.g.} \cite[\S52.C]{Arno89},
due to underdetermination (choice of control in \eqref{eq-sysO}, multi-valued right-hand side in \eqref{eq-sysM}). 

\emph{Informally,} ``\ref{it-thconv-a}''  states that any solution of the average system is the
limit of solutions of fast oscillating systems with well chosen controls and ``\ref{it-thconv-b}'' states that,
conversely, any limit of solutions of oscillating systems, with arbitrary controls,
is a solution of the average control system. There is an estimate on the error in ``\ref{it-thconv-a}'' but not in
``\ref{it-thconv-b}'' because some sequences may converge slower than others.


\refstepcounter{theorem}\paragraph{Remark~\thetheorem}\label{rmk-thconv} 
One may consider systems
that are \emph{affine} instead of linear in the control by adding a drift vector field $\GG_0(t/\varepsilon,x)$ to
\eqref{eq-sysO}. Then, in the average control system, $\EE(x)$ is replaced by $\overline{\GG}_0(x)+\EE(x)$, with
$\overline{\GG}_0(x)=\frac1{2\pi}\int_0^{2\pi}\GG_0(\theta,x)\mathrm{d}\theta$. By a straightforward extension, convergence does hold for these systems too.


\smallskip
In the proof of Theorem~\ref{th-conv-sysM}, the following technical lemma is needed.

\begin{lemma}
  \label{lem-chvar}
Let $\varepsilon>0$ and $a<b$ be real numbers and 
$\widehat{u}:[a-2\pi\varepsilon,b]\times S^1\to\RR^m$ be measurable. One has the following identity
(see \S\ref{sec-not-S1} for the notation $\mu(.)$)~:
\begin{equation}
  \label{eq:052}
  \iint_{\substack{\\[1.2ex]\theta\in S^1\\a\leq s\leq b}}\hspace{-1.5ex}
  \GG(\theta,x(s))\;\widehat{u}(s,\theta)\; \mathrm{d}\theta\, \mathrm{d}s
  =
  \iint_{\substack{\\[1.2ex] \theta\in S^1\\a\leq s\leq b}}\hspace{-1.2ex}
\GG(\frac s\varepsilon,\III(s))\; \widehat{u}(s+\varepsilon\mu(\theta),\frac s\varepsilon)\; \mathrm{d}\theta\, \mathrm{d}s
+\Delta_\varepsilon
\end{equation}

\smallskip

\noindent
with\\[-3.3em]
\begin{eqnarray}
\nonumber
\Delta_\varepsilon &=&\iint_{T^a_\varepsilon} \GG(\frac s\varepsilon,\III(s+\varepsilon\mu(\theta)))\, 
\widehat{u}(s+\varepsilon\mu(\theta),\frac s\varepsilon)\, \mathrm{d}\theta\,\mathrm{d}s
\\\nonumber&&-\iint_{T^b_\varepsilon} \GG(\frac s\varepsilon,\III(s+\varepsilon\mu(\theta)))\,
\widehat{u}(s+\varepsilon\mu(\theta),\frac s\varepsilon)\, \mathrm{d}\theta\,\mathrm{d}s
\\ 
\label{eq:053}
&&+
\iint_{\substack{\theta\in S^1\\a\leq s\leq b}}
\left[\GG(\frac s\varepsilon,\III(s+\varepsilon\mu(\theta)))-\GG(\frac s\varepsilon,\III(s))\right] \widehat{u}(s+\varepsilon\mu(\theta),\frac s\varepsilon)\, \mathrm{d}\theta\,\mathrm{d}s
\end{eqnarray}
and the set $T^a_\varepsilon$ defined by $T^a_\varepsilon\!=\!\{(s,\theta),\,\theta\!\in\! S^1\!,\,a-\varepsilon\,\mu(\theta)\leq s\leq a\}$ and $T^b_\varepsilon$ accordingly.
\end{lemma}
\begin{proof}
Thanks to the change of variables
$\theta=\tau/\varepsilon\!\!\mod\!2\pi$, $s=\tau+\varepsilon\,\mu(\phi)$, with $\mu(\theta)$ as defined in \S\ref{sec-not-S1}, the left-hand side of (\ref{eq:052}) is equal to
$$
\iint_{\substack{\phi\in S^1\\a-\varepsilon\phi\leq \tau\leq
    b-\varepsilon\mu(\phi)}}\GG(\frac\tau\varepsilon,\III(\tau+\varepsilon\mu(\phi)))\, \widehat{u}(\tau+\varepsilon\mu(\phi),\frac\tau\varepsilon)\,
\mathrm{d}\tau\, \mathrm{d}\phi\ .
$$
Keeping the names $(s,\theta)$ instead of $(\tau,\phi)$, one gets (\ref{eq:052}), the correcting term
$\Delta_\varepsilon$ coming from the modified domain of integration and argument of $\III$.
\hfill\end{proof}

{\em Proof of Theorem \ref{th-conv-sysM}, point \ref{it-thconv-a}}. 
Consider a solution $\III_0:[0,T]\to\vari^n$ of \eqref{eq-sysM}. 
By Proposition~\ref{lem-meas} there exists $\widehat{u}_0\in L^\infty([0,T]\times S^1,\RR^m)$, $\|\widehat{u}_0\|_\infty\leq1$ satisfying~\eqref{eq:038}.
For $\varepsilon>0$, define $\overline{u}_\varepsilon(.)\in L^\infty([0,T],\RR^m)$ 
by (see \S\ref{sec-not-S1} for notations):
 \begin{equation}
 \label{eq:ueps-bar}
 \overline{u}_\varepsilon(t)=\frac1{2\pi}\int_0^{2\pi} \widehat{u}_0(t+\varepsilon\,\mu(\theta),\frac{t}{\varepsilon})\,\mathrm{d}\theta\;,
 \end{equation}
where $\widehat{u}_0$ is prolonged by zero outside $[0,T]$: 
$\widehat{u}_0(t+\varepsilon\,\mu(\theta),\frac{t}{\varepsilon})=0$ if $t+\varepsilon\,\mu(\theta)>T$.
Let us prove that this construction of $\overline{u}_\varepsilon$ satisfies the two announced properties.

\emph{Step 1. Let us first assume that $\vari$ is an open subset of $\RR^n$ and $\mathcal{G}$ is zero outside a compact subset of $\vari$.}
Then $\mathcal{G}(\theta,x)$ is a $n\times m$ matrix for all $(\theta,x)$ and, denoting by $\|.\|$ the Euclidean norm for vectors and the operator norm for matrices, there are global constants $\Lip \GG$ and $\sup\GG$ such that, for all $x,x',\theta$ in $\vari\times\vari\times S^1$,
\begin{equation}
  \label{eq:step1-2}
  \left\|\GG(\theta,x)-\GG(\theta,x')\right\|\leq\bigl(\Lip \GG\bigr)\|x-x'\|
\,,\ \ \left\|\GG(\theta,x)\right\|\leq\bigl(\sup \GG\bigr)\ .
\end{equation}
Let $b$ be a non-negative constant and consider, for each $\varepsilon>0$, a solution $\III_\varepsilon(.)$ of~\eqref{eq-sysO} with control 
$u=\overline{u}_\varepsilon(t)$ and initial condition $\III_\varepsilon(0)$ such that
\begin{equation}
  \label{eq:033}
  \|\III_\varepsilon(0)-\III_0(0)\|\leq b\,\varepsilon\;.
\end{equation}
In fact $b=0$ in the theorem itself, but we need a nonzero $b$ in step 2.
By definition, expanding $\overline{u}_\varepsilon(s)$ as in (\ref{eq:ueps-bar}) and using Lemma~\ref{lem-chvar}, one has
\begin{eqnarray}
\nonumber
 \III_\varepsilon(t) & = & 
\III_\varepsilon(0)+\frac1{2\pi}\int_0^t \! \int_0^{2\pi}\GG(\frac s\varepsilon,\III_\varepsilon(s))\, \widehat{u}_0(s-\varepsilon\,\mu(\theta),\frac{s}{\varepsilon})\, \mathrm{d}\theta\, \mathrm{d}s\,,
\\ \label{eq:031}
& = &
\III_\varepsilon(0)+\frac1{2\pi}\left(\int_0^t \! \int_0^{2\pi}
\GG(\theta,\III_\varepsilon(s))\,\widehat{u}_0(s,\theta)
\, \mathrm{d}\theta\, \mathrm{d}s
\;-\;\Delta_\varepsilon\right)
\end{eqnarray}
with $\Delta_\varepsilon$ given by (\ref{eq:053}), that satisfies $\|\Delta_\varepsilon\|\leq 4\pi^2\bigl(\Lip
\GG\bigr)\left(1+T\sup\GG\right)\varepsilon$ because, in particular, $\|\widehat{u}_0\|\leq1$, $|\varepsilon\mu(\theta)|<2\pi\varepsilon$ and
 \begin{displaymath}
\|\left(
 \GG(\frac{s}\varepsilon,\III_\varepsilon(s))
 -
 \GG(\frac{s}\varepsilon,\III_\varepsilon(s+\varepsilon\,\mu(\theta))) \right) \widehat{u}_0(s+\varepsilon\,\mu(\theta),\frac{s}{\varepsilon})
\| 
  \leq  2\pi\,\bigl(\Lip \GG\bigr) \,\bigl(\sup \GG\bigr) \,\varepsilon\;.
\end{displaymath}
Using (\ref{eq:033}), (\ref{eq:031}) the bound on $\|\Delta_\varepsilon\|$ and the relation
\begin{eqnarray*}
 \III_0(t) & = & 
\III_0(0)+\frac1{2\pi}\int_0^t \! \int_0^{2\pi} \GG(\theta,\III_0(s))\, \widehat{u}_0(s,\theta)\, \mathrm{d}\theta\, \mathrm{d}s\,,
\end{eqnarray*}
one gets
\begin{displaymath}
 \|\III_\varepsilon(t)-\III_0(t)\| \leq  \left(b+2\pi\,\bigl(\Lip \GG\bigr) \,\bigl(1+T\sup \GG\bigr)\right)\varepsilon + \bigl(\Lip \GG\bigr) \int_0^t \|\III_\varepsilon(s)-\III_0(s)\|\,\mathrm{d}s
 \end{displaymath}
for all $t$ in $[0,T]$, and finally, by Gronwall lemma,
\begin{equation}
\label{eq:gron}
 \|\III_\varepsilon(t)-\III_0(t)\|\leq
\left[b+
2\pi\bigl(\Lip \GG\bigr)\left(1+T\sup\GG\right)\right]\, e^{T\Lip \GG}\,\varepsilon
 \end{equation}
for all $t$ in $[0,T]$ and $\varepsilon$ in $[0,\varepsilon_0]$.
This proves the theorem if $\vari$ is an open subset of $\RR^n$ and $\mathcal{G}$ is zero outside a compact subset,
with an explicit constant $c$ corresponding to the distance $d$ defined from the Euclidean norm and with $\varepsilon_0=+\infty$.

\emph{Step 2. General case.}
Let $\III_\varepsilon(.)$ be the solution of~\eqref{eq-sysO} with control $u=\overline{u}_\varepsilon(t)$ defined in
(\ref{eq:ueps-bar}) from $\widehat{u}_0$ and with initial condition $\III_\varepsilon(0)=\III_0(0)$; it is not necessarily defined on $[0,T]$ but may have a maximum interval of definition $[0,T_\varepsilon)$ with $T_\varepsilon<T$.
Let $\widetilde{T}\in[0,T]$ be the supremum of the set of numbers $\tau\in[0,T]$ such that, 
for some $\varepsilon_0$ and some $c$, that may depend on $\tau$, the solution $\III_\varepsilon(.)$ is defined on $[0,\tau]$ 
and satisfies $d(\III_\varepsilon(t),\III_0(0))<c\,\varepsilon$ for all $t\in[0,\tau]$ and $\varepsilon\in[0,\varepsilon_0]$. 
Let us prove by contradiction that $\widetilde{T}=T$. This will end the 
proof of Theorem \ref{th-conv-sysM}, point \ref{it-thconv-a}.

Assume $\widetilde{T}<T$, and let
\\- $\mathcal{O}$ be a coordinate neighborhood of $x_0(\widetilde{T})$, 
\\- $\alpha>0$ be such that $0<\widetilde{T}-\alpha<\widetilde{T}+\alpha\leq T$ and $x_0([\widetilde{T}-\alpha,\widetilde{T}+\alpha])\subset\mathcal{O}$, 
\\- $c>0$, $\varepsilon_0>0$ be such that 
$d(\III_\varepsilon(t),\III_0(t))<c\,\varepsilon$ for all $t\in[0,\widetilde{T}-\alpha]$ and
$\varepsilon\in[0,\varepsilon_0]$.
\\
Taking $\varepsilon_0$ possibly smaller, one also has $\III_\varepsilon(\widetilde{T}-\alpha)\in\mathcal{O}$ for
$\varepsilon<\varepsilon_0$.
Let $\Set{K}$ be a compact neighborhood of $x_0([\widetilde{T}-\alpha,\widetilde{T}+\alpha])$ contained in $\mathcal{O}$,
$\Set{K}'$ a compact neighborhood of $\Set{K}$ contained in $\mathcal{O}$, and $\rho:\vari\to[0,1]$ a smooth map, zero outside 
$\Set{K}'$ and constant equal to 1 in $\Set{K}$. Defining $\GG_\rho$ by $\GG_\rho(\theta,x)=\rho(x)\GG(\theta,x)$,
let us apply Step 1 in coordinates in $\mathcal{O}$, with $\GG_\rho$ instead of $\GG$ and
$[\widetilde{T}-\alpha,\widetilde{T}+\alpha]$ instead of $[0,T]$.
Call $\III_0^\rho$ (resp. $\III_\varepsilon^\rho$, $\varepsilon>0$) the solution of \eqref{eq:038}
(resp. of~\eqref{eq-sysO} with control $u=\overline{u}_\varepsilon(t)$), replacing $\GG$ by $\GG_\rho$, 
with initial condition $\III_\varepsilon^\rho(\widetilde{T}-\alpha)=\III_\varepsilon(\widetilde{T}-\alpha)$,
$\varepsilon\geq0$.
One clearly has, as in \eqref{eq:033}, 
$\|\III_\varepsilon^\rho(\widetilde{T}-\alpha)-\III_0^\rho(\widetilde{T}-\alpha)\|<b\,\varepsilon$
with $b$ deduced from $c$ via Lipschitz equivalence of the distance $d$ and the Euclidean norm in coordinates
(see \S\ref{sec-not-d}); then Step 1 provides $\varepsilon_0'>0$ such that, by \eqref{eq:gron},  the inequality
\begin{equation}
\label{eq:gron2}
 \|\III_\varepsilon^\rho(t)-\III_0^\rho(t)\|\leq
\left[b+
2\pi\bigl(\Lip \GG_\rho\bigr)\left(1+2\alpha\sup\GG_\rho\right)\right]\, e^{2\alpha\Lip \GG_\rho}\,\varepsilon
 \end{equation}
is valid for $t\in[\widetilde{T}-\alpha,\widetilde{T}+\alpha]$ and $\varepsilon\in[0,\varepsilon_0']$.
Possibly choosing a smaller $\varepsilon_0'$, this implies that
$x_\varepsilon([\widetilde{T}-\alpha,\widetilde{T}+\alpha])\subset\Set{K}$ for $\varepsilon<\varepsilon_0'$; 
since $\GG$ coincides with $\GG_\rho$ in $\Set{K}$, the conclusion holds for $\III_\varepsilon$ and $\GG$
as well as for $\III_\varepsilon^\rho$ and $\GG_\rho$ if $\varepsilon$ is no larger than $\varepsilon_0'$.
We have shown that,  for all $\varepsilon<\varepsilon_0'$, the solution  $\III_\varepsilon$ is defined on $[0,\widetilde{T}+\alpha]$
and satisfies $d(\III_\varepsilon(t)-\III_0(t))\leq c'\varepsilon$ for $t$ in $[0,\widetilde{T}+\alpha]$ where
$c'$ is larger than $c$ and than a bound deduced from (\ref{eq:gron2}) and from Lipschitz equivalence of $d$ with the Euclidean distance.
This contradicts the definition of $\widetilde{T}$. 
\hfill\endproof

{\em Proof of Theorem \ref{th-conv-sysM}, point \ref{it-thconv-b}}. 
Since $\GG$ is bounded on $S^1\times\Set{K}$ (one may cover $\Set{K}$ with a finite number of coordinate charts and
define this bound in coordinates), the maps $x_n(.)$ have a common Lipschitz constant and
the sequence $(x_n(.))$ is equi-continuous, hence compact by Ascoli-Arzela Theorem: one may extract a uniformly convergent sub-sequence.
Still denoting by $(x_n(.))_{n\in\NN}$ such a converging sub-sequence and by $\III^*(.)$ its (uniform) limit,
we need to prove that this limit is a solution of \eqref{eq-sysM}.

Define, for each $n$, $\widehat{u}_n:[0,T]\times S^1\to\RR^m$ by
\begin{equation}
  \label{eq:036}
  \widehat{u}_n(t,\theta)=u_n(\,\beta_{n}(t,\theta)\,)\,,
\end{equation}
where $u_n(.)\in L^\infty([0,T],\RR^m)$ is associated to $\III_n(.)$ according to the assumption of the theorem and
where the map $\beta_n:[0,T]\times S^1\to\RR$ is defined by
\begin{equation}
  \label{eq:035}
  t-2\pi\varepsilon_n<\beta_n(t,\theta)\leq t
\;,\ \ \ 
\frac{\beta_n(t,\theta)}{\varepsilon_n}\equiv\theta\mod 2\pi\ .
\end{equation}
Clearly $\widehat{u}_n$ is in $L^\infty([0,T]\times S^1,\RR^m)$ and $\|\widehat{u}_n\|_\infty\leq1$.
Hence, after possibly extracting a sub-sequence, $(\widehat{u}_n)$ converges in the weak-$*$ topology to some $\widehat{u}^*$.
Let us prove that, for almost all $t\in[0,T]$,
\begin{equation}
  \label{eq:034}
  \dot{x}^*(t)=\frac1{2\pi}\int_0^{2\pi}\GG(\theta,x^*(t))\,\widehat{u}^*(t,\theta)\mathrm{d}\theta\ .
\end{equation}
Let $\widetilde{T}\in[0,T]$ be the supremum of the set of numbers $\tau\in[0,T]$ such that this is true for
almost all $t$ in $[0,\tau]$, and let us prove by contradiction that $\widetilde{T}=T$. 

Assume $\widetilde{T}<T$, and let $\mathcal{O}$ be a coordinate neighborhood of $x_0(\widetilde{T})$ and
$\alpha$ be such that $0\!<\!\widetilde{T}-\alpha\!<\!\widetilde{T}+\alpha\!\leq\! T$ and $x_0([\widetilde{T}-\alpha,\widetilde{T}+\alpha])\!\subset\!\mathcal{O}$.
Uniform convergence implies $x_n ([\widetilde{T}-\alpha,\widetilde{T}+\alpha])\!\subset\!\mathcal{O}$ for $n$
large enough and then, in coordinates, for $t\in[\widetilde{T}-\alpha,\widetilde{T}+\alpha]$,
\begin{eqnarray}
\label{eq:037}
 \III_n(t)-\III_n(\widetilde{T}-\alpha) & = & \int_{\widetilde{T}-\alpha}^{t}\GG(\frac{s}{\varepsilon_n},\III_n(s))\, u_n(s)\, \mathrm{d}s \,.
\end{eqnarray}
From (\ref{eq:035}), one has $\beta_n(s+{\varepsilon_n}\theta,\frac s{\varepsilon_n})=s$, hence, from
(\ref{eq:036}), $\widehat{u}_n(s+{\varepsilon_n}\theta,\frac s{\varepsilon_n})=u_n(s)$ for all $\theta\in S^1$, $ s
\in\RR$; using this in Lemma~\ref{lem-chvar}, one has
\begin{eqnarray*}
  {\textstyle \frac1{2\pi}}\iint_{\substack{\\[1.1ex]\widetilde{T}-\alpha\leq s\leq t\\\theta\in S^1}}
\hspace{-1.5em}
\GG(\theta,\III_n(s))\widehat{u}_n(s,\theta)\, \mathrm{d}\theta\, \mathrm{d}s
&=&
 {\textstyle \frac1{2\pi}}\iint_{\substack{\\[1.1ex]\widetilde{T}-\alpha\leq s\leq t\\\theta\in S^1}}
\hspace{-.8em}
\GG(\frac{s}{\varepsilon_n},\III_n(s))\,u_n(s)\, \mathrm{d}\theta\, \mathrm{d}s
\;+\Delta_{\varepsilon_n}\;.
\end{eqnarray*}
Since the integral in the right-hand side ---whose integrand does not depend on $\theta$--- is equal to the right-hand
side of  (\ref{eq:037}), one gets, using uniform convergence of $\III_n$ to
$\III^*$, weak convergence of $\widehat{u}_n$ to $\widehat{u}^*$ and convergence of $\Delta_{\varepsilon_n}$ to zero, 
\begin{eqnarray*}
 \III^*(t)-\III^*(\widetilde{T}-\alpha) & = & \frac1{2\pi}\int_{\widetilde{T}-\alpha}^t \int_0^{2\pi}
\GG(\theta,\III^*(s))\widehat{u}^*(s,\theta)\, \mathrm{d}\theta\, \mathrm{d}s\ ,
\end{eqnarray*}
for $t$ in $[\widetilde{T}-\alpha,\widetilde{T}+\alpha]$, and finally that (\ref{eq:034}) hold for almost all
$t$ in $[0,\widetilde{T}+\alpha]$, thus contradicting the definition of $\widetilde{T}$. 
\hfill\endproof


\subsection{Dimension of the velocity set $\EE(\III)$}
\label{sec-propav}
Recall that, for a convex subset $C$ of a linear space, containing the origin,
its linear hull is the smallest linear subspace that contains $C$,
the interior of $C$ \emph{in its linear hull} is always nonempty, and
$\dim C$ is the dimension of this linear hull. 

Viewing $\frac{\partial^j \GG}{\partial\theta^j}(\theta,\III)$ as a linear map $\RR^m\to\mathrm{T}_\III\vari$ (see
Remark~\ref{rmk-G-oper}), and denoting by $\Sigma$ a sum of linear subspaces of $T_\III\vari$, define the
integer $r(\theta,\III)$ by: 
    \begin{equation}
      \label{eq:def-r}
      r(\theta,\III)
      =\dim\Bigl(\sum_{j\in\NN}\Range\frac{\partial^j \GG}{\partial \theta^j}(\theta,\III)\Bigr)\,.
    \end{equation}
It is also the rank
    of the collection of vectors $\frac{\partial^j \GG_i}{\partial \theta^j}(\theta,\III)\in\mathrm{T}_\III\vari$, $1\leq i\leq m$, $j\geq0$.

In the following proposition, and it is the sole place where this property is used, ``system (\ref{eq-sysO}) is real
analytic with respect to $\theta$'' means that the vector fields $\GG_i$ are real analytic with respect to $\theta$ for fixed
$\III$ (while being smooth with respect to $(\theta,\III)$).

\begin{proposition}
  \label{lem-dim} 
  \begin{remunerate}
  \item\label{lemdim-it1} The linear hull of $\EE(\III)$ satisfies the following two properties for all $\III$ in $\vari$,
    where the inclusion (\ref{eq:dim3}) is an equality if (\ref{eq-sysO}) is real analytic with respect to $\theta$: 
    \begin{align}
      \label{eq:dim1}
      &\LinearHull\EE(\III)\ =\ \sum_{\theta\in S^1}\Range\GG(\theta,\III)\;,
      \\
      \label{eq:dim3}
      &\LinearHull\EE(\III)
      \ \supset\ 
      \sum_{j\in\NN}\Range\frac{\partial^j \GG}{\partial \theta^j}(\theta,\III)
      \hspace{2em}\text{for all $\theta\in S^1$}\;.
    \end{align}
  \item\label{lemdim-it2} If $r(\theta,\III)=n$ for at least one $\theta$ in $S^1$, then $\EE(\III)$ has a nonempty
    interior in $T_\III\vari$, i.e. $\dim\EE(\III)=n$. 
  \item\label{lemdim-it3} If the system (\ref{eq-sysO}) is real analytic with respect to $\theta$, then $r(\theta,\III)$
    does not depend on $\theta$ and $r(\theta,\III)=\dim\EE(\III)$. 
  \end{remunerate}
\end{proposition}
\begin{proof}
  If $p$ is in ${\Range\GG(\theta,\III)}^{\,\perp}$ for all $\theta$,
  then any $v=\overline{\GG}(\III,\mathscr{U})$ in $\EE(\III)$ satisfies $\langle p,v\rangle=0$ because
  $\langle p,\GG(\theta,\III)\mathscr{U}(\theta)\rangle$ is identically zero on $[0,2\pi]$.
  Conversely, let $p$ be in $\EE(\III)^{\perp}$, and consider $v=\overline{\mathcal{G}}(\III,\mathscr{U}^*_{p,\III})\in\EE(\III)$;
  then $\langle p,v\rangle=0$ implies $\langle p,\GG(\theta,\III)\rangle=0$, 
  i.e. $p\in{\Range\GG(\theta,\III)}^{\,\perp}$ for all $\theta$.
We have proved the identity 
$\EE(\III)^{\perp}=\bigcap_{\theta\in S^1}\left(\Range\GG(\theta,\III)\right)^\perp$, hence (\ref{eq:dim1}).
\\
If $p$ is in ${\Range\GG(\phi,\III)}^{\,\perp}$ for all $\phi$, differentiating $\left\langle p\,,\GG(\phi,\III)\right\rangle\!=\!0$
  with respect to $\phi$ yields $\left\langle
    \,p\,,\partial^j\GG/\partial\phi^j(\III,\phi)\right\rangle=0$, $j\in\NN$; we have proved,
  taking $\phi\!=\!\theta$, the inclusion
$\bigcap_{\phi\in S^1} \left(\Range\GG(\phi,\III)\right)^\perp \subset
    \bigcap_{j\in\NN}\bigl(\Range\frac{\partial^j\GG}{\partial \theta^j}(\theta,\III)\bigr)^\perp$, hence \eqref{eq:dim3}.
To prove the reverse inclusion in the real analytic case, fix $\theta\in S^1$ and 
  $p\in\bigcap_{j\in\NN} { \frac{\partial^j \GG}{\partial \theta^j}(\theta,\III) } ^\perp$, 
  and consider the real analytic mapping $S^1\to(\RR^m)^*$, $\phi\mapsto \left\langle p\,,\GG(\phi,\III)\right\rangle $; 
  the assumption on $p$ implies that this map vanishes for $\phi=\theta$, as well as its derivatives at all orders,
  hence it is identically zero: $p\in\bigcap_{\phi\in S^1}\left(\Range\GG(\phi,\III)\right)^\perp$.
This ends the proof of Point \ref{lemdim-it1}.
Point \ref{lemdim-it2} is an easy consequence and Point \ref{lemdim-it3} is classical.
\hfill\end{proof}

\subsection{Further properties in the full rank case}
\label{sec-fullrank}

We now assume that the mapping $\mathcal{G}$ in (\ref{eq-sysO}) is such that the rank $r(\theta,\III)$ defined by (\ref{eq:def-r}) is maximal:
\vspace{-1ex}
\begin{equation}
  \label{eq:fullrank}
  r(\theta,\III)= n\ \ \text{for all $\III$ in  $\vari$ and $\theta$ in $S^1$.}
\end{equation}

\subsubsection{Controllability}

Condition \eqref{eq:fullrank} is strongly related to controllability of the linear approximation of 
(\ref{eq-sysO}) around equilibria, i.e. around solutions where $\III$ is constant and $u$ is identically
zero. Indeed, the linear approximation of the time-varying nonlinear system (take $\varepsilon=1$ in (\ref{eq-sysO})):
\vspace{-1ex}
\begin{equation}
  \label{eq:sys01}
  \dot{\III}=\mathcal{G}(t,\III)u
\vspace{-1ex}
\end{equation}
 around the equilibrium $\III=\III_1$ is the time-varying linear system $\dot{\xi}=\mathcal{G}(t,\III_1)u$; according to
\cite[p.614]{Kail80}, it is ``controllable with impulsive controls at any time'' if and only if $r(t,\III_1)=n$ for all
$t$. If this is true at all points $\III_1$ then all end-point mappings are submersions around zero controls; we shall
need the following more precise result: 

\smallskip

\begin{proposition}
  \label{prop-controllability}
  Assume that (\ref{eq:fullrank}) holds.
  \begin{remunerate}
  \item\label{item-cont1} 
For all $\III_1\in\vari$ and $T\!>\!0$, there exist
a coordinate neighborhood $\mathcal{W}$ of $\III_1$ (the ball $\mathcal{B}$ below refers to the Euclidean norm in these coordinates),
positive constants $\alpha_0,c_3$, 
and, for all $y\in\mathcal{W}$, a smooth
    map $\chi_y:\mathcal{B}(y,\alpha_0)\to  L^\infty([0,T],\RR^m)$ with Lipschitz constant $c_3$,
    which is a right inverse of the end-point mapping of (\ref{eq:sys01}) on $[0,T]$ starting from $y$, 
    i.e. for all $y_f\in \mathcal{B}(y,\alpha_0)$, the control $\chi_y(y_f):[0,T]\to\RR^m$ is such that the solution of
      $\dot{\III}=\mathcal{G}(t,\III)\chi_y(y_f)(t),\;\III(0)=y$ satisfies $\III(T)=y_f$.
  \item\label{item-cont2} For all $\varepsilon>0$, the system (\ref{eq-sysO}) is fully controllable, i.e. there exists, for any $\varepsilon>0$ and any two point $\III_0,\III_1$ in $\vari$, a
    time $T$ and a measurable control $u:[0,T]\to B^m$ such that the solution of  (\ref{eq-sysO}) with $\III(0)=\III_0$ satisfies $\III(T)=\III_1$.
  \end{remunerate}
\end{proposition}
\begin{proof}
Let $E_y:L^\infty([0,T],\RR^m)\to\vari$ be the end-point mapping with starting point $y$.
Condition \eqref{eq:fullrank} implies that the derivative of $E_{\III_1}$ at the zero control has rank $n$; hence there
exists an $n$-dimensional subspace $V$ of $L^\infty([0,T],\RR^m)$
such that the restriction of $E_{\III_1}$, and hence of $E_y$ for $y$ close enough, to $V$ is a local diffeomorphism at zero; the
$\chi_y$'s are the local inverses of these local diffeomorphisms; they depend smoothly on $y$, hence the common $\alpha_0$ and $c_3$ in Point 1.

This implies that the reachable set from any point at any positive or negative time contains a neighborhood of this
point; a classical argument then tells us that the reachable set from a
point $\III_0$ is $\vari$, assumed to be connected, for it is both open (obvious) and closed (if $\bar{\III}$ is
in the closure of the reachable set, some points in the reachable set can be reached in negative time, hence $\bar{\III}$
can be reached from $\III_0$).
\hfill\end{proof}

\medskip

Let us now turn to the average system  (\ref{eq-sysM}). From $H:\mathrm{T}^*M\to[0,+\infty)$ defined
by~\eqref{eq:Hamiltonian}, we define $N:\mathrm{T}\vari\to[0,+\infty]$ by
\vspace{-1ex}
\begin{equation}
  \label{eq:5}
  N(\III,v)=\max_{p\in \mathrm{T}_\III^*M,\;H(\III,p)\leq1} \langle p,v\rangle\ .
\end{equation}
\begin{proposition}
  \label{prop-Finsler}
  Assume the rank condition (\ref{eq:fullrank}).
  \begin{remunerate}
  \item\label{it-PropFinsler-1} For all $\III\in\vari$, $H(\III,.)$ 
         defines a norm on the cotangent space $\mathrm{T}_\III^*M$,
    its dual norm on the tangent space $\mathrm{T}_\III M$ is $N(\III,.)$, and 
    $\EE(\III)$ is the unit ball for $N(\III,.)$, i.e. \\$\EE(\III)=\{v\in \mathrm{T}_\III M,\;N(\III,v)\leq1\}$.
  \item\label{it-PropFinsler-3} System (\ref{eq-sysM}) is fully controllable, i.e. there exists, for any points $\III_0,\III_1$ in $\vari$, a
    time $T$ and a solution $\III(.):[0,T]\to\vari$ of  (\ref{eq-sysM}) such that $\III(0)=\III_0$, $\III(T)=\III_1$.
  \end{remunerate}
\end{proposition}
\begin{proof}
  From (\ref{eq:Hamiltonian}), $H(\III,p)=0$ implies $\langle p,\GG(\theta,\III)\rangle=0$  for all $\theta$ and,
differentiating with respect to $\theta$ and using (\ref{eq:fullrank}), this implies $p=0$; this makes $p\mapsto H(\III,p)$
a norm, the other properties being straightforward. Hence $N$ given by \eqref{eq:5} is finite for any $(\III,v)$ and it is, by
definition, the dual norm of $H(\III,.)$; $\EE(\III)$ is its unit ball by (\ref{eq:EH}) in Proposition~\ref{lem-E-H}.
To prove Point~\ref{it-PropFinsler-3}, take a continuously differentiable curve $\gamma:[0,1]\to\vari$ such that $\gamma(0)=\III_0$ and
$\gamma(1)=\III_1$ and $\sigma:[0,T]\to[0,1]$ for some $T>0$, differentiable, such that
$$t\geq\int_0^{\sigma(t)}N(\gamma(s),\frac{\mathrm{d}\gamma}{\mathrm{d}s}(s)) \mathrm{d}s$$
($N$ and $H$ are obviously continuous),
then $t\mapsto \III(t)=\gamma(\sigma(t))$ is a solution of (\ref{eq-sysM}) such that $\III(0)=\III_0$ and $\III(T)=\III_1$.
\hfill\end{proof}

\subsubsection{On the differentiability of $H$}
\label{sec-H-diff}
It is clear that $H$, given by \eqref{eq:Hamiltonian}, is as smooth as $\GG$
on $\mathrm{T}^*\vari\setminus \widetilde{\mathcal{Z}}$ with 
\begin{eqnarray}
  \label{eq-Ztilde}
  \widetilde{\mathcal{Z}}&=&\{(\III,p)\in \mathrm{T}^*\vari,\,\exists\theta
  \in S^1,\,\langle p , \GG(\theta,\III) \rangle=0\}\ .
\end{eqnarray}
Unfortunately,
$\widetilde{\mathcal{Z}}$ is not empty in general: it is generically a $2n-m+1$ dimensional
submanifold of $\mathrm{T}^*\vari$.
One however has the following result, valid also at these points.

\smallskip

\begin{theorem}
  \label{prop-HnormC1}
  If condition (\ref{eq:fullrank}) holds, $H^2$ is continuously differentiable.
\end{theorem}

\smallskip

It is stated for $H^2\!:(\III,p)\mapsto H(\III,p)^2$, because $H$ itself,
homogeneous of degree 1 with respect to $p$, cannot be differentiable on
$\{p=0\}$, that coincides with $\{H(\III,p)=0\}$ by Proposition~\ref{prop-Finsler} item~\ref{it-PropFinsler-1}.

The map $H$ fails in general to be twice differentiable on $\widetilde{\mathcal{Z}}$.
We have the following estimate of the of the modulus of continuity of its first derivative, that even fails to be
Lipschitz continuous.
Its main consequence is Theorem~\ref{th-ham-flow}.

\smallskip

\begin{theorem}
\label{th-ham-LipLog}
Assume that the rank condition (\ref{eq:fullrank}) holds and that
\begin{romannum}
  \item\label{it-lip-1} for $(\III,p)\!\in\mathrm{T}^*\vari$, $p\neq0$, there is at most one $\theta\!\in S^1$ such
    that  $\langle p , \GG(\theta,\III) \rangle\!=\!0$, and $\displaystyle\langle p ,
    \frac{\partial\GG}{\partial\theta}(\theta,\III) \rangle$ does not vanish at the same point,
  \item\label{it-lip-2} for all $(\theta,\III)\in S^1\times\vari$, one has $\rank\GG(\theta, \III)=m$,
\end{romannum}
\vspace{.8ex}then any point $(\bar \III,\bar p)$  has a constant $c$ and a coordinate neighborhood in $\mathrm{T}^*\vari$ such that for all
$X$ and $Y$ in $\RR^{2n}$, coordinates of points in the neighborhood,
\begin{equation}
  \label{eq:Lip-Log}
  \|\mathrm{d}H(Y)-\mathrm{d}H(X)\|\leq c\,\|X-Y\|\, \ln\frac1{\|X-Y\|}\ .
\vspace{-1ex}
\end{equation}
\end{theorem}
In the left-hand side, $\|.\|$ stands for the operator norm
in coordinates, see \S\ref{sec-not-d}.

\smallskip


\refstepcounter{theorem}\paragraph{Remark~\thetheorem\ (Finsler geometry)}\label{rmk-Finsler}
If $H^2$ was an least \emph{twice} continuously differentiable, with a positive definite Hessian with respect to $p$, so
would be $N^2$ (see \eqref{eq:5}), and it would define a (reversible) \emph{Finsler metric}~\cite{Bao-Che-She00} on $\vari$.
The lack of differentiability calls for further developments.


\medskip

Before proving these theorems, we state a more generic result, whose notations are totally independent from the rest of the paper.
Its proof is in the appendix.
\begin{proposition}
    \label{prop-C1-LipLog} 
Let $\mathsf{d}$ be a positive integer, $\varid$ an open subset of $\RR^{\mathsf{d}}$, $\mathsf{V}\!\!:S^1\times\varid\to\RR^m$ a smooth map ($C^\infty$),
$\widetilde{\mathcal{Z}}$ the subset of $\varid$ where $\mathsf{V}$ vanishes for some $\theta$ and $\mathsf{H}\!\!:\varid\!\to[0,+\infty)$
the average of the norm of $\mathsf{V}$:
\vspace{-.9ex}
\begin{equation}
  \label{eq:H-Ztilde}
  \mathsf{H}(X)=\frac1{2\pi}\int_0^{2\pi}\!\!\|\mathsf{V}(\theta,X)\|\,\mathrm{d}\theta\ ,\ \ \   
  \widetilde{\mathcal{Z}}=\{X\in \varid,\,\exists\theta
  \in S^1,\,\mathsf{V}(\theta,X)=0\}\ .
\end{equation}
\begin{remunerate}\vspace{-2ex}
\item \label{point-C1} Assume that, for all  $X$ in $\varid$, \vspace{-1.4ex}
  \begin{equation}
    \label{eq:8}
    \text{the set 
      $\{\theta\in S^1\,,\;\mathsf{V}(\theta,X)=0\}$
      has measure zero in $S^1$}.
  \end{equation}
$\ $\\[-4ex]
  Then $\mathsf{H}$ is continuously differentiable and, for all $X$, 
  \begin{equation}
    \label{eq:H'}
    \mathrm{d}\mathsf{H}(X).h=
    \frac1{2\pi}\int_0^{2\pi}\scal{\frac{\partial\mathsf{V}}{\partial X}(\theta,X).h}{\frac{\mathsf{V}(\theta,X)}{\|\mathsf{V}(\theta,X)\|}}
    \,\mathrm{d}\theta\ .
  \end{equation}
\item \label{point-LipLog}
Let $\mathsf{V}$ satisfy the following assumptions:

\vspace{-1.2\baselineskip}

\begin{equation}
\label{eq:A123}
\left.
\begin{array}{cl}
  \text{\textup{(a)}}&
  \mathsf{V}(\theta,X)=0\, \text{ for all } (\theta,X)\in S^1\times\varid \text{ such that }
  \mathsf{V}(\theta,X)=0,
\\[.2ex]
  \text{\textup{(b)}}&
  \text{for any $X\in\varid$, there is at most one $\theta$ such that $\mathsf{V}(\theta,X)=0$},
\\[.2ex]
  \text{\textup{(c)}}&
  \text{$\mathsf{V}$ and $\partial \mathsf{V}/\partial\theta$ do not vanish simultaneously},
\end{array}
\right\}
\end{equation}

\vspace{-.5\baselineskip}

 \noindent
and let $\bar{X}$ be in $\widetilde{\mathcal{Z}}$. There is a neighborhood $\ouv$ of $\bar{X}$ in $\varid$, and a
constant $K>0$ such that, for all $X,Y$ in $\ouv$,\vspace{-1ex} 
\begin{eqnarray}
  \label{eq:LipLog}
  \Bigl\| \mathrm{d}\mathsf{H}(X)-\mathrm{d}\mathsf{H}(Y) \Bigr\| &\leq&
K\; \|X-Y\|\;\ln\frac1{\|X-Y\|}\ .
\end{eqnarray}\vspace{-2ex}
\end{remunerate}
\end{proposition}
\noindent In the left-hand side of \eqref{eq:LipLog}, $\|.\|$ stands for the operator norm, see \S\ref{sec-not-d}.

\smallskip

{\em Proof of Theorem~\ref{prop-HnormC1}.}
This is a local property.
  We operate in coordinates and   apply
  Proposition~\ref{prop-C1-LipLog} (Point \ref{point-C1}) with $\mathsf{d}=2n$, $\varid$
  a neighborhood of a point where $p\neq0$, $X=(\III,p)\in\RR^{2n}$ and $\mathsf{V}(\theta,X)=\langle
  p,\GG(\theta,\III)\rangle$.
  The rank condition implies that derivatives of all orders of the map
  $\theta\mapsto \mathsf{V}(\theta,X)$ never vanish at the same point, so that its zeroes are isolated and the set $\{\theta\in
  S^1\,,\;\mathsf{V}(\theta,X)=0\}$ is finite and \textit{a fortiori} has measure zero; hence $H$ is continuously differentiable
  outside $\{p=0\}$.  Since $0\leq H(\III,p)\leq k\|p\|$ for some local constant $k$ the derivative of $H^2$ is zero at all
  points $(\III,0)$ and, since (\ref{eq:H'}) implies that the norm of $\mathrm{d}H(\III,p)$ at neighboring points where
  $p\neq0$ is bounded, the derivative of $H^2$ at these points tends to zero as $p\to0$.  $H^2$ is therefore
  continuously differentiable everywhere.
\hfill\endproof

{\em Proof of Theorem~\ref{th-ham-LipLog}.}
Smoothness outside $\widetilde{\mathcal{Z}}$ is obvious from the expression \eqref{eq:Hamiltonian} of $H$; 
inequality \eqref{eq:Lip-Log} is a consequence of Proposition~\ref{prop-C1-LipLog} (Point \ref{point-LipLog}), applied with $\mathsf{d}=2n$, 
$X=(\III,p)\in\RR^{2n}$, $\mathsf{V}(\theta,X)=\langle p,\GG(\theta,\III)\rangle$ and $\varid$ a neighborhood of a point of
$\widetilde{\mathcal{Z}}\setminus\{p=0\}$; it is clear that points (i) and (ii) imply the three conditions~\eqref{eq:A123}.
\hfill\endproof

\subsection{Application to the minimum time problem}
\label{sec-Tmin}

Fix  two points $\III_0, \III_1$ in $\vari$ and
consider the time optimal problem associated to \eqref{eq-sysO} for $\varepsilon>0$:
\begin{equation}
   \label{eq-pb-Tmin}
(\mathcal{P}_\varepsilon),\;\varepsilon>0:\qquad\left.
\begin{array}{l}
  \dot{\III}(t)=\GG(t/\varepsilon,\III(t))u(t), \ u(t)\in B^m,\;t\in[0,T],
\\
   \III(0)=\III_0,\, \III(T)=\III_1
\end{array}
\right\}
\ \  \min T \,,
\end{equation}
and  the time optimal problem associated to the average system: 
\begin{equation}
\label{eq-pb-TminMoy}
(\mathcal{P}_0):\qquad\left.
\begin{array}{l}
\dot{\III}(t) \in \EE(\III(t)), \;t\in[0,T],
\\
   \III(0)=\III_0,\, \III(T)=\III_1
\end{array}
\right\}
\ \  \min T \,.
\end{equation}
Call $T_\varepsilon(\III_0,\III_1)$ the minimum time for $(\mathcal{P}_\varepsilon)$, $\varepsilon>0$ and $T_0(\III_0,\III_1)$ the one for $(\mathcal{P}_0)$;
when no confusion arises, we write $T_\varepsilon$ and $T_0$.

Let us develop \eqref{eq-pb-Tmin}--\eqref{eq-pb-TminMoy}:
concerning \eqref{eq-pb-Tmin}, 
$T_\varepsilon$ is the infimum of the set of $T$'s such that there is 
an admissible control $u(.) \!\!:[0,T]\!\to\!B^m\!$, and 
$\III(.) \!\!:[0,T]\!\to \!\vari$ 
satisfying $\III(0)=\III_0$, $\III(T)=\III_1$ and 
$\dot{\III}(t)=\GG(t/\varepsilon,\III(t))u(t)$ for almost all $t$; 
Proposition \ref{prop-controllability}, point
\ref{item-cont2} implies that this set is nonempty, hence $T_\varepsilon$ is finite.
Concerning \eqref{eq-pb-TminMoy}, 
$T_0$ is the infimum of the set of $T$'s such that there is 
$\III(.)\!:[0,T]\to\vari$ 
satisfying $\III(0)=\III_0$, $\III(T)=\III_1$ and
$\dot{\III}(t)\in\EE(\III(t))$ for almost all $t$,
$T_0$ is finite from Proposition \ref{prop-Finsler}, point \ref{it-PropFinsler-3}.
A \emph{solution} to $(\mathcal{P}_\varepsilon)$ (resp. to $(\mathcal{P}_0)$) is $\III(.), u(.)$ (resp. $\III(.)$) as
above with $T=T_\varepsilon$ (resp. $T=T_0$). 
In general, the minimum $T_\varepsilon$ or $T_0$ need not be reached, i.e. there need not be a solution.

\smallskip

\begin{lemma}
  \label{lem-TT}
Assume the rank condition (\ref{eq:fullrank}).

1. There is a neighborhood $\mathcal{W}$ of any $x_1$ and two constants $\alpha_0>0$ and $C_3>0$ such that, for all $y$
    in $\mathcal{W}$, $T_\varepsilon(y,x_1)\leq2\pi\varepsilon+C_3 d(x_1,y)$.

2.  For any $x_0,x_1',x_1$ in $\vari$, one has $T_\varepsilon(x_0,x_1)\leq T_\varepsilon(x_0,x'_1)+T_\varepsilon(x'_1,x_1)+2\pi\varepsilon$.
\end{lemma}
\smallskip
\begin{proof}
  Apply Proposition~\ref{prop-controllability}, point \ref{item-cont1} with $T=2\pi$, using as a distance in
  $\mathcal{W}$ the Euclidean norm in some coordinates: for any two points $y,y'$ in
  $\mathcal{W}$ such that $\|y-y'\|\leq\alpha_0$, there is a control defined on $[0,2\pi]$, with $L^\infty$ norm smaller
  than $c_3\|y-y'\|$ that brings $y'$ at time 0 to $y$ at time $2\pi$ for system (\ref{eq:sys01}); rescaling time and
  control by $\varepsilon$ yields, if $c_3\|y-y'\|\leq\varepsilon$, a control with $L^\infty$ norm less than 1 that
  brings $y'$ at time 0 to $y$ at time $2\pi\varepsilon$ for system  (\ref{eq-sysO}) and hence, by concatenating
  controls and using periodicity of $\mathcal{G}$, for any positive integer $k$, a control  with $L^\infty$ norm less than 1 that
  brings $y'$ at time 0 to $y$ at time $2k\pi\varepsilon$ for system  (\ref{eq-sysO}) if $c_3\|y-y'\|\leq
  k\varepsilon$. In other words, $T_\varepsilon(y',y)\leq 2\pi(\varepsilon+c_3\|y-y'\|)$. Take $y'=x_1$ and $2\pi c_3/C_3$ 
  the ratio between the Euclidean norm and the distance $d$; this proves point 1.
     Point 2 follows from using periodicity of $\mathcal{G}$ and concatenating controls while inserting a zero control between
  time $T_\varepsilon(y',y)$ and the next multiple of $2\pi$.
\hfill\end{proof}

\smallskip

\begin{theorem}[limit of minimum time]
\label{prop-limT}
Assume the rank condition (\ref{eq:fullrank}).
\begin{remunerate}
\item\label{limT-it1} $T_\varepsilon$ is bounded as $\varepsilon\to0$ and
$\  \limsup_{\varepsilon\to0} T_\varepsilon\leq T_0\ .$
\item\label{limT-it2} If, for $\varepsilon>0$ small enough, each $(\mathcal{P}_\varepsilon)$ has a solution
$\III_\varepsilon:[0,T_\varepsilon]\to\vari$ and there exists a compact $\Set{K}\subset\vari$ such that $\III_\varepsilon([0,T_\varepsilon])\subset\Set{K}$
for all $\varepsilon>0$ small enough, then 
all accumulation points of the compact family $(\III_\varepsilon(.))_{\varepsilon>0}$ in $C^0([0,T_0],\vari)$ are solutions of $(\mathcal{P}_0)$ and
$\ \ \displaystyle \lim_{\varepsilon\to0}T_\varepsilon=T_0\ .$
\end{remunerate}
\end{theorem}
\begin{proof}
  Consider a minimizing sequence for problem $(\mathcal{P}_0)$, i.e. solutions $\III^k:[0,T_0+\beta_k]\to\vari$ of the
  average system (\ref{eq-sysM}) with $(\beta_k)$ a sequence of positive numbers that tends to zero and $x^k(0)=x_0$,
  $x^k(T_0+\beta_k)=x_1$ for all $k$. For each $x^k$, there is, according to Theorem~\ref{th-conv-sysM}, a family
  $(\,x^k_\varepsilon(.)\,)_{\varepsilon>0}$ such that each $x^k_\varepsilon(.)$ is a solution of (\ref{eq-sysO}) with
  $x^k_\varepsilon(0)=x_0$ and $d(x^k_\varepsilon(t),x^k(t))\leq c_1\varepsilon$ for all $t$ in $[0,T_0+\beta_k]$. In
  particular $d(x^k_\varepsilon(T_0+\beta_k),x^k(t)_1)\leq c_1\varepsilon$.  Now, from Lemma~\ref{lem-TT},
  $T_\varepsilon(\,x^k_\varepsilon(T_0+\beta_k)\,,x_1)\leq(2\pi+c_1C_3)\varepsilon$; hence, from the second point of
  that lemma (with $x'_1=x^k_\varepsilon(T_0+\beta_k)\,$), one has $T_\varepsilon = T_\varepsilon(x_0,x_1)\leq
  T_0+\beta_k+(4\pi+c_1C_3)\varepsilon$ and, letting $k$ go to infinity, $T_\varepsilon\leq
  T_0+(4\pi+c_1C_3)\varepsilon$; this implies Point~\ref{limT-it1}. Let us turn to point 2.

  Extend $x_\varepsilon$ on $[0,\overline{T}]$, with $\overline{T}$ an upperbound of
  $T_\varepsilon$, by taking $x_\varepsilon(t)=x_1$ for $t$ in $[T_\varepsilon,\overline{T}]$. Any sequence
  $(x_{\varepsilon_k}(.))_{k\in\NN}$ with $\lim\varepsilon_k=0$ is compact in $C^0([0,\overline{T}],\vari)$: take a
  convergent subsequence such that $T_{\varepsilon_k}$ also converges to some $T^*$. 
  The uniform limit goes through $x_0$ at time 0 and $x_1$ at time $T^*$ and is, by Theorem~\ref{th-conv-sysM},
  a solution of the average system (\ref{eq-sysM}), hence $T^*\geq T_0$ by definition of $T_0$. This, together with Point~\ref{limT-it1}, implies
  Point~\ref{limT-it2} because $T^*$ can be any accumulation point of
  $(T_\varepsilon)$ as $\varepsilon\to 0$.
\hfill\end{proof}

\medskip

Let us now write the Pontryagin Maximum Principle~\cite{Pont-Bol-Gam-M74}  both for $(\mathcal{P}_\varepsilon)$, $\varepsilon>0$ and for
$(\mathcal{P}_0)$ and see how they are related.

The \emph{extremals} of problem $(\mathcal{P}_\varepsilon)$, $\varepsilon>0$, are absolutely continuous maps $t\mapsto(x(t),p(t))$
solution 
to
 \begin{equation}
 \label{eq-sys-maxH}
     \dot{p} =  -  \frac{\partial H_\varepsilon}{\partial x}\,,\ \;
     \dot{x} =   \frac{\partial H_\varepsilon}{\partial p}
\qquad\text{with}\qquad H_\varepsilon(t,p,x)=\|\langle p,\mathcal{G}(t/\varepsilon,x)\rangle\|\,,
\end{equation}
whose right-hand side is discontinuous on $\mathscr{S}_\varepsilon=\{(x,p,t),\;\langle
p,\mathcal{G}(t/\varepsilon,x)\rangle=0\}$ (the ``switching surface''), where it is in fact not defined.

The \emph{extremals} of 
$(\mathcal{P}_0)$ are absolutely continuous $t\mapsto(x(t),p(t))$ solution to
 \begin{equation}
 \label{eq-sys-maxHmoy}
     \dot{p} =  -  \frac{\partial H}{\partial x}\,,\ \ \ 
     \dot{x} =   \frac{\partial H}{\partial p}\;.
\end{equation}
with $H$ given by (\ref{eq:Hamiltonian}). The right-hand sides are continuous according to Theorem~\ref{prop-HnormC1}.

\medskip

\begin{theorem}
\label{th-ham}
If an absolutely continuous map $t \mapsto \overline{x}(t)$ defined on $[0,\overline{T}]$ is a solution of
$(\mathcal{P}_\varepsilon)$, $\varepsilon>0$, (resp. of $(\mathcal{P}_0)$), 
then there exists $t\mapsto \overline{p}(t)$ defined on $[0,\overline{T}]$ such that $t\mapsto(\overline{p}(t),\overline{\III}(t))$ is an
extremal of $(\mathcal{P}_\varepsilon)$, $\varepsilon>0$ (resp. of $(\mathcal{P}_0)$).
\end{theorem}
\begin{proof}
Problem~$(\mathcal{P}_\varepsilon)$, $\varepsilon>0$ deals with a classical smooth control system; according to
\cite{Pont-Bol-Gam-M74,Agra-Sac:04}, the pseudo-Hamiltonian is $h(t,x,p,u)=\langle p , \mathcal{G}(t/\varepsilon,x)\, u
\rangle$; an extremal is a curve on the co-tangent bundle solution, in local coordinates, of:
 \begin{equation}
   \label{eq-ext-M}
   \begin{array}{lcl}
     \dot{p} & = &-\frac{\partial h}{\partial x}(t,x,p,u^*)\;=\; - \langle p , \frac{\partial\mathcal{G}}{\partial x} u^* \rangle ,\\
     \dot{x} & = & \frac{\partial h}{\partial p}(t,x,p,u^*)\;=\;  \mathcal{G}\, u^* , \\
   \end{array}
 \end{equation}
with $u^*(t)$ a control that maximizes the pseudo-Hamiltonian for almost all time; it is defined by
$u^*=\frac{\langle p,\mathcal{G}\rangle}{\|\langle p,\mathcal{G}\rangle\|}$ if $\langle p,\mathcal{G}(t/\varepsilon,x)\rangle\neq0$;
the maximized Hamiltonian $H_\varepsilon(t,p,x)=\max_u h(t,x,p,u)$ is the one in \eqref{eq-sys-maxH}, and
(\ref{eq-ext-M}) is then the differential equation
\eqref{eq-sys-maxH}, whose right-hand side is discontinuous at points where $\langle p,\mathcal{G}(t/\varepsilon,x)\rangle$ vanishes.

Let us now turn to 
$(\mathcal{P}_0)$.
Since the set of admissible velocities is not a priori smooth with respect to the state variable we 
use a
non-smooth version of the Pontryagin maximum principle for differential inclusions, that we recall for self-containedness: 
\begin{center}
  \begin{minipage}[t]{0.93\linewidth} {\it
    \emph{Theorem~9.1 in \cite[Chapter 4]{Clar:98}:}\hspace{0.6ex} if $\dot{x}\in\EE(x)$ is a locally Lipschitz
    differential inclusion and $t \mapsto \overline{x}(t)$ is an absolutely continuous function defined on
    $[0,\overline{T}]$ solution to the problem~\eqref{eq-pb-TminMoy}, then there exists $t\mapsto \overline{p}(t)$
    defined on $[0,\overline{T}]$ such that
    $(-\dot{\overline{p}},\dot{\overline{x}}) \in \partial_C H(\overline{x},\overline{p})$ for almost all $t\in[0,\overline{T}]$
    with $H(x,p)=\max_{v\in\EE(x)}\langle p,v\rangle$ and $\partial_C H$ the generalized gradient of $H$. }
  \end{minipage}
\end{center}
\vspace{0.6ex}
The set-valued map $\EE(.)$ in \eqref{eq:calE} is indeed locally Lipschitz:
in local coordinates , for $x_1,x_2$ in
$\RR^n$, denoting by $\delta$ the Hausdorff distance between two sets, one has:
\begin{eqnarray*}
\delta\left(\EE(x_1),\EE(x_2)\right) 
& = & \max 
\left\{
 \sup_{v_1 \in \EE(x_1)} \inf_{v_2 \in \EE(x_2)} \|v_1-v_2\| ,
  \sup_{v_2 \in \EE(x_2)} \inf_{v_1 \in \EE(x_1)} \|v_1-v_2\|
 \right\} \\
 & = & 
\max_{\|\mathscr{U}_1\|_{\infty}\leq 1} \min_{\|\mathscr{U}_2\|_{\infty}\leq 1} \| \overline{\GG}(x_1,\mathscr{U}_1)- \overline{\GG}(x_2,\mathscr{U}_2)\| 
  \ \leq\  \Lip \GG \,\, \|x_1-x_2\| .
\end{eqnarray*}
According to (\ref{eq:EHbis}), the Hamiltonian $H$ defined in the above quoted theorem coincides with the map $H$ defined 
in (\ref{eq:Hamiltonian}).
\hfill\end{proof}
\smallskip


\refstepcounter{theorem}\paragraph{Remark~\thetheorem}\label{rmk-HamCom}
This result and Theorem~\ref{prop-limT} have two interpretations: 
\\1. They prove that the operations of
\emph{averaging} and \emph{computing the Hamiltonian for the minimum time problem} commute. Indeed, the Hamiltonian $H$
was obtained by applying the maximum principle to problem (\ref{eq-pb-TminMoy}), i.e. minimum time for the 
average system (\ref{eq-sysM}), but it also the average of the one in (\ref{eq-sys-maxH})
with respect to the fast variable.
\\2. They prove indirectly an averaging result for the minimum time control problem \eqref{eq-pb-TminMoy}; the
averaging techniques in \cite{Chap87} do not apply to minimum time for they require smoothness of the Hamiltonian,
while averaging is used in \cite{Gef-Epe:97,Geff97th} for minimum time with only partial theoretical justifications but
numerical evidence of efficiency.


\medskip


Let us now focus on the
differential equations (\ref{eq-sys-maxHmoy}) that govern the extremals of $(\mathcal{P}_0)$.
It is of great importance to know whether it defines a Hamiltonian flow on $\mathrm{T}^*\vari$, \textit{i.e.} whether
solutions trough all initial conditions are unique or not.
Its right-hand side is continuous because, from Theorem~\ref{prop-HnormC1}, $H$ is continuously differentiable; this
ensures existence of solutions. We saw that $H$ is smooth ($C^\infty$) on
$\mathrm{T}^*\vari\setminus \widetilde{\mathcal{Z}}$ (see \eqref{eq-Ztilde}), hence solutions through points outside
$\widetilde{\mathcal{Z}}$ are always unique.
The following result 
gives uniqueness of solutions even on $\widetilde{\mathcal{Z}}$ in the less degenerated case possible.

\begin{theorem}[Hamiltonian flow for $(\mathcal{P}_0)$]
\label{th-ham-flow}
Assume that the rank condition (\ref{eq:fullrank}) holds, as well as conditions (i) and (ii) in Theorem~\ref{th-ham-LipLog}.
Then the differential equation \eqref{eq-sys-maxHmoy} has a unique solution from any initial condition.
\end{theorem}

\begin{proof}
For an autonomous ODE $\dot{z}=f(z)$ in a finite dimensional space,
where $f$ satisfies
$\|f(z_1)-f(z_2)\|\leq\omega(\|z_1-z_2\|)$ with $\omega\!:[0,+\infty)\!\to\!
[0,+\infty)$ non-decreasing, Kamke uniqueness Theorem \cite[chap. \!III, Th. \!6.1]{Hart:82} states that uniqueness of solutions holds if
$\int_0^\alpha\!\frac{\mathrm{d}u}{\omega(u)}\!=\!+\infty$ for arbitrarily small $\alpha>0$.
From Theorem~\ref{th-ham-LipLog}, we are in this case with $\omega(u)=c\,u\ln(1/u)$, and $\int\frac{\mathrm{d}u}{\omega(u)}=-\frac1c\,\ln\ln(1/u)$. 
\hfill\end{proof}

\smallskip

Proving existence of a flow for \eqref{eq-sys-maxHmoy} in more general situations (weaker sufficient condition) than
this theorem is an interesting program to
be pursued. However, it turns out to be applicable to
the control of orbit transfer with low thrust, see~\S\ref{sec-2body}.

Point (ii) is very mild and only states that the control vector fields are linearly independent. 
Point (i) is more
artificial: the fact that  $\langle p , \GG(\theta,x) \rangle=0$ has at most one solution $\theta$ has to be checked by
hand, while the fact that $\partial\GG/\partial\theta$ does not vanish at the same time is equivalent to the
rank condition\vspace{-1ex}
\begin{equation}
  \label{eq:rank-2}
  \rank
\left\{\GG(\theta,\III),\frac{\partial\GG}{\partial \theta}(\theta,\III)\right\}
=\dim\Bigl(\Range\GG(\theta,x)+\Range\frac{\partial\GG}{\partial \theta}(\theta,x)\Bigr)=n\ .
\vspace{-1ex}
\end{equation}
It is true for the Kepler problem and used in \cite{Cail-Noa01} to show that the discontinuities in \eqref{eq-sys-maxH} are
always ``$\pi$-singularities'', i.e. the control $u^*$ switches to its opposite.

\section{Kepler control systems}
\label{sec-kep}

We call \emph{Kepler control system} with small control a family of control system on $S^1\times\vari$
of the form
\begin{equation}
  \label{eq:Kep}
(\mathcal{K}_\varepsilon)\ \ \left\{
  \begin{array}{rcl}
    \dot{\theta}&=&\pulse(\theta,x)+g(\theta,x)\,v
    \\
    \dot{x}&=&G(\theta,x)\,v
  \end{array}\right.
\ ,\ \ \ \|v\|\leq\varepsilon\ ,
\end{equation}
where $G$ and $g$ can be viewed, with the same convention is in \eqref{eq-sysO}, as $n\times m$ and $1\times m$ matrices
smoothly depending on $(\theta,x)$ and $\pulse$ is a smooth function $S^1\times\vari\to\RR$
that remains larger than a strictly positive constant:\vspace{-1.5ex}
\begin{equation}
  \label{eq:pulsemin}
  \pulse(\theta,x)\geq k_{\pulse}>0\ \ \ \forall(\theta,x)\in S^1\times\vari\ .
\vspace{-1.5ex}
\end{equation}
In fact, this is an affine control system on $S^1\times\vari$\vspace{-1ex}
\begin{equation}
  \label{eq:Kep-champs}
  \dot{\xi}=f_0(\xi)+\sum_{i=1}^{m}v_i f_i(\xi)
\vspace{-1ex}
\end{equation}
with $\xi=(\theta,x)$, $f_0=\pulse\,\frac\partial{\partial\theta}$ and, for $1\leq i\leq m$, the smooth vector field
$f_i$ is represented by the $i$\textsuperscript{th} column of the matrix notations $G$ and $g$. If,
in \eqref{eq:Kep-champs}, one only assumes that, all solutions of $\dot{\xi}=f_0(\xi)$ are periodic, 
additional conditions are needed for the orbits
to induce a nice foliation that splits the state manifold into a product $\vari\times S^1$.

\subsection{Relation with fast oscillating systems}
\label{sec-kep-trsf}
For a solution $t\!\mapsto\!(\theta(t),\III(t))$ of $(\mathcal{K}_\varepsilon)$ in \eqref{eq:Kep},
let $\Theta(t)$ be the cumulated angle \textit{i.e.} $\Theta(.)$ is
continuous $[0,T]\to\RR$ with $\Theta(t)\!\equiv\!\theta(t)$ mod  $2\pi$ for all $t$ and 
$\Theta(0)\in[0,2\pi)$, and define a new ``time''
\begin{equation}
  \label{eq:13}
  \epsthet=\mathscr{R}(t)\stackrel\Delta=\varepsilon\left(\Theta(t)-\Theta(0)\right)\,.
\end{equation}
Taking $\varepsilon_0$ small enough so that $|\pulse (\theta,x)+\varepsilon\,g(\theta,x)\,u|>k_\pulse/2$ for $x\in\Set{K}$,
$\|u\|\leq1$, $\varepsilon<\varepsilon_0$, one has $\mathrm{d}\mathscr{R}/\mathrm{d}t>\varepsilon\,k_\pulse/2$ hence
$\mathscr{R}$ is strictly increasing and one-to-one, and 
\begin{equation}
  \label{eq:22}
  \frac{k_\pulse}{2}\varepsilon\,t\leq\mathscr{R}(t)\leq k_\pulse\,\varepsilon\,t
\ \ \text{with } k_\pulse =\sup_{S^1\times\Set{K}}\pulse+\varepsilon_0 \sup_{S^1\times\Set{K}}\|g\|
\,.
\end{equation}

Then $\lambda\mapsto \widetilde{x}(\epsthet)\!=\!x(\mathscr{R}^{-1}(\epsthet))$ is a solution of the system
\begin{equation}
  \label{eq:001}
  (\widetilde{\Sigma}_{\theta_0,\varepsilon})\ \ \ \ 
  \frac{\mathrm{d}\widetilde{x}}{\mathrm{d}\epsthet}=
  \frac{G (\theta_0+\frac\epsthet\varepsilon,\widetilde{x})\,\widehat{u}} 
  {\pulse (\theta_0+\frac\epsthet\varepsilon,\widetilde{x})\;+\;\varepsilon\,g(\theta_0+\frac\epsthet\varepsilon,\widetilde{x})\,\widehat{u}}
  \ ,\ \ \ \|\widehat{u}\|\leq 1\ ,
\end{equation}
associated with the control $\epsthet\mapsto\widehat{u}(\epsthet)
\!={v(\mathscr{R}^{-1}(\epsthet))}/{\varepsilon}$.
Except for the term $\varepsilon g \widehat{u}$ in the denominator, this is a
fast oscillating system \eqref{eq-sysO} with $\GG=G/\pulse$.
We now apply \S\ref{sec-fast}.

\subsection{Average control system}
\label{sec-kep-aver}
The definition uses $\overline{\pulse}$ defined by
\begin{equation}
  \label{eq:20}
  \frac1{\overline{\pulse}(x)}=
\frac1{2\pi}\int_0^{2\pi}\frac{\mathrm{d}\theta}{\pulse(\theta,x)}\ .
\end{equation}
\begin{definition}[Average control system of Kepler control systems]
  \label{def-aver-Kep}
The average control system of the Kepler control system (\ref{eq:Kep}) is the differential inclusion
\begin{equation}
  \label{eq:21}
  \dot\III\in \EE(x)
\end{equation}
with $\EE$ defined by (\ref{eq:calE}) using 
$\overline{\GG}:\vari\times L^\infty(S^1,\RR^m)\to\mathrm{T}\vari$ defined by 
\begin{equation}
  \label{eq:defGbarKep}
  \overline{\GG}(x,\mathscr{U}) \;=\;
\overline{\pulse}(x)\,
\frac{1}{2\pi}\!\!
\int_0^{2\pi}\!\frac{G(\theta,x)}{\pulse(\theta,x)}\,\mathscr{U}(\theta)\mathrm{d}\theta
\end{equation}
instead of \eqref{eq:defGbar}.
Solutions are defined as in Definition~\ref{def-aver}.
\end{definition}


\refstepcounter{theorem}\paragraph{Remark~\thetheorem}\label{rmk-KepFast}
This is almost Definition~\ref{def-aver} applied to \eqref{eq:001}, which is equivalent to \eqref{eq:Kep}
via time changes, \emph{except}:
\begin{romannum}
\item the term $\varepsilon g \widehat{u}$ in the denominator of \eqref{eq:001} has been discarded,
\item the right-hand side has been multiplied by $\overline{\pulse}(x)$. 
\end{romannum}


\subsection{Convergence Theorem} The counterpart of Theorem~\ref{th-conv-sysM} is:

\begin{theorem}[Convergence for Kepler control systems]
\label{th-conv-sysM-Kep}

\textup{1.}
  Let $\III_0(.):[0,T]\to\vari$ be an arbitrary solution of~\eqref{eq:21} and $\theta^0\in S^1$.
  There exist a family of measurable functions $\overline{u}_\varepsilon(.):[0,T]\to B^m$, indexed by $\varepsilon>0$,
  and positive constants $c,\varepsilon_0$, such that, if $t\mapsto(\theta_\varepsilon(t),\III_{\varepsilon}(t))$ is the solution of \eqref{eq:Kep}
   with control $u=\overline{u}_\varepsilon(t)$ and initial condition
   $(\theta_\varepsilon(0),\III_{\varepsilon}(0))=(\theta^0,\III_0(0))$,  it is defined on $[0,T/\varepsilon]$
   for $\varepsilon$ smaller than $\varepsilon_0$ and
\begin{equation}
  \label{eq:CVkep}
  \mathit{d}(\,\III_{\varepsilon}(t)\,,\,\III_0(\varepsilon t)\,)<c\,\varepsilon\,,\ \ t\in[0,\frac T\varepsilon]\,,\ \ 
0<\varepsilon<\varepsilon_0\,,
\end{equation}
thus
$\tau\mapsto\III_{\varepsilon}(\tau/\varepsilon)$ converges uniformly on $[0,T]$ to $\tau\mapsto\III_0(\tau)$ when $\varepsilon$ tends to zero.

\textup{2.}
    Let $\Set{K}$ be a compact subset of $\vari$, $(\varepsilon_n)_{n\in\NN}$ a decreasing sequence of positive real numbers converging to zero,
and $\bigl(\theta_n(.),\III_{n}(.)\bigr):\,[0,T/\varepsilon_n]\to S^1\times\Set{K}$ 
a solution of system~(\ref{eq:Kep}) for each $n$, with $\varepsilon=\varepsilon_n$ and some control
$u=u_n(t)$, $u_n(.)\in L^\infty([0,T/\varepsilon_n],\RR^m)$, $\|u_n(.)\|_\infty\leq1$.
Then the sequence $\bigl(\tKep\mapsto(\III_{n}(\tKep/\varepsilon_n))\bigr)_{n\in\NN}$ is compact for the topology of uniform convergence on
$[0,T]$ and the limit of any converging sub-sequence is a solution $\III^*(.)$ of the average differential inclusion
\eqref{eq:21}.
\end{theorem}
\begin{proof}
We assume that $\vari$ is $\RR^n$, $d$ the Euclidean distance and all vector fields have a common compact support,
hence all maps share a global Lipschitz constant and a global bound; by ``a constant'', we mean a number that
depends only on these bounds and Lipschitz constants.
It is left to the reader to check that, as for the proof of Theorem~\ref{th-conv-sysM}, the present proof extends to $\vari$ with any distance $d$ described in
\S\ref{sec-not-d}. 

Let $\tKep\mapsto\III_0(\tKep)$ be a solution of~\eqref{eq:21} on $[0,T]$.  
Define $\mathscr{P}(.)$ by
\begin{equation}
  \label{eq:18}
  \mathscr{P}(\tKep) = 
  \int_0^\tKep \overline{\pulse}\left(\III_0(\mathfrak{t})\right) \mathrm{d}\mathfrak{t}
\end{equation}
and $\widehat{x}_0(.)$ by 
$\widehat{x}_0(\epsthet)= x_0(\mathscr{P}^{-1}(\epsthet))$. The latter 
is a solution on $[0,\mathscr{P}(T)]$ of
\begin{equation}
  \label{eq:24}
  \frac{\mathrm{d}\widehat{x}_0}{\mathrm{d}\epsthet}\in \frac1{\overline{\pulse}(\widehat{x}_0)}\mathscr{E}(\widehat{x}_0)
\end{equation}
with $\mathscr{E}$ defined by \eqref{eq:calE} and \eqref{eq:defGbarKep}.
This is the average system (in the sense of Definition \ref{def-aver}) of the fast oscillating control system
\begin{equation}
  \label{eq:001bb}
  (\widehat{\Sigma}_{\theta_0,\varepsilon})\ \ \ \ 
  \frac{\mathrm{d}\widehat{x}}{\mathrm{d}\epsthet}=
  \frac{G (\theta_0+\frac\epsthet\varepsilon,\widehat{x})\,\widehat{u}} {\pulse
    (\theta_0+\frac\epsthet\varepsilon,\widehat{x})}
  \ ,\ \ \ \|\widehat{u}\|\leq 1\;.
\end{equation}Theorem~\ref{th-conv-sysM} (Point 1) yields a family of controls $\widehat{u}_\varepsilon$ such that the solutions
$\widehat{x}_\varepsilon(.)$ of $(\widehat{\Sigma}_{\theta_0,\varepsilon})$ with initial condition
$\widehat{x}_0(0)$ and control $\widehat{u}_\varepsilon$ converge to $\widehat{x}_0(.)$ uniformly:
\begin{equation}
  \label{eq:14}
  d(\widehat{x}_\varepsilon (\epsthet),\widehat{x}_0 (\epsthet))\leq c'\,\varepsilon\ \ \text{for all } \epsthet \in[0,\mathscr{P}(T)]
\end{equation}
for some constant $c'$. For each $\varepsilon$, let
$\widetilde{x}_\varepsilon(.)$ be the solution of $(\widetilde{\Sigma}_{\theta_0,\varepsilon})$ ---~see \eqref{eq:001}~--- with
same initial condition and same control.
Since \eqref{eq:001} can be re-written as
\begin{equation}
  \label{eq:1}
  \frac{\mathrm{d}\widetilde{x}}{\mathrm{d}\epsthet}=\Bigl(1- \frac{\varepsilon\,
     g(\theta_0+\frac\epsthet\varepsilon,\widetilde{x})\,\widehat{u}\;
     }{
     \pulse(\theta_0+\frac\epsthet\varepsilon,\widetilde{x})+\varepsilon \,g(\theta_0+\frac\epsthet\varepsilon,\widetilde{x})\,\widehat{u}
     }\Bigr)
\frac{G (\theta_0+\frac\epsthet\varepsilon,\widetilde{x})\,\widehat{u}} {\pulse
  (\theta_0+\frac\epsthet\varepsilon,\widetilde{x})}\,,
\end{equation}
the norm of the difference between the right-hand sides of $(\widetilde{\Sigma}_{\theta_0,\varepsilon})$ and $(\widehat{\Sigma}_{\theta_0,\varepsilon})$
is bounded by $k\,\varepsilon$ for some constant $k>0$; 
classical theorems on
smooth dependence of solutions on ``parameters'' yield some constant $c''$ such that
\begin{equation}
  \label{eq:26}
  d(\widetilde{x}_\varepsilon (\epsthet),\widehat{x}_\varepsilon (\epsthet))\leq c''\,\varepsilon\ \ \text{for all } \epsthet \in[0,\mathscr{P}(T)]\,.
\end{equation}
Then define
\begin{equation}
  \label{eq:17}
  t=\mathfrak{T}(\epsthet)=
  \frac1\varepsilon
  \int_0^\epsthet
  \frac{
    \mathrm{d}\ell
  }{
    \pulse(\theta_0+\frac\ell\varepsilon,\widetilde{x}_\varepsilon(\ell)) + \varepsilon \,g(\theta_0+\frac\ell\varepsilon,\widetilde{x}_\varepsilon(\ell))\,\widehat{u}_\varepsilon(\ell)
  }
\end{equation}
and the controls $t\mapsto \overline{u}_\varepsilon(t)$ by
$\widehat{u}_\varepsilon(\epsthet)=\overline{u}_\varepsilon(\mathfrak{T}(\epsthet))$; the solutions $x_\varepsilon(.)$
of \eqref{eq:Kep} with these controls are given by
$\widetilde{x}_\varepsilon(\epsthet)={x}_\varepsilon(\mathfrak{T}(\epsthet))$, and one therefore has
\begin{equation}
  \label{eq:27}
  d(\III_\varepsilon(\mathfrak{T}\!\circ\!\mathscr{P}(\tKep)),x_0(\tKep))<(c'+c'')\varepsilon,\ \ \tKep\in[0,T]\,.
\end{equation}

Now, on the one hand, \eqref{eq:17} yields
\begin{align}
  \label{eq:6}
  \mathfrak{T}(\mathscr{P}(\tKep))=&\;
  \frac1\varepsilon
  \int_0^{\mathscr{P}(\tKep)}
  \frac{
    \mathrm{d}\ell
  }{
    \pulse(\theta_0+\frac\ell\varepsilon,\widehat{x}_\varepsilon(\ell)) 
  }
+\rho
\\\nonumber
  \mbox{with}\ \ \ \rho=&\;\frac1\varepsilon\int_0^{\mathscr{P}(\tKep)}
\left(
  \frac{
    \widehat{\pulse}(\ell) -\widetilde{\pulse}(\ell)
  }{
    \widehat{\pulse}(\ell)
  }
+
  \frac{
    \varepsilon \,g(\ell)\,\widehat{u}_\varepsilon(\ell)
  }{
    \widetilde{\pulse}(\ell) + \varepsilon \,g(\ell)\,\widehat{u}_\varepsilon(\ell)
  }
\right)
\frac{\mathrm{d}\ell}{\widetilde{\pulse}(\ell) }
\end{align}
where $\widetilde{\pulse}(\ell)$ stands for $\pulse(\theta_0+\frac\ell\varepsilon,\widetilde{x}_\varepsilon(\ell))$,
$\widehat{\pulse}(\ell)$ stands for
$\pulse(\theta_0+\frac\ell\varepsilon,\widehat{x}_\varepsilon(\ell))$, and $g(\ell)$ stands for
$g(\theta_0+\frac\ell\varepsilon,\widetilde{x}_\varepsilon(\ell))$;
using \eqref{eq:26} and Lipschitz continuity of $\pulse$ to bound the first term in the integral, this implies that
$|\rho|$ is bounded by a constant.
On the other hand, one has, according to \eqref{eq:18},
$\tau
=
\int_0^{\mathscr{P}(\tKep)} \!\!
\mathrm{d}\epsthet  /  \overline{\pulse}(\widehat{x}_0(\epsthet))$. Developing $\overline{\pulse}$ according to its
definition \eqref{eq:20}, in which we add $\theta_0\!+\!\frac{\epsthet}{\varepsilon}$ to $\theta$ without changing the
integral due to periodicity,
$\tau$ is also equal to
$\frac1{2\pi}
\iint_{\theta\in S^1,\;0\leq \epsthet\leq \mathscr{P}(\tKep)}
\mathrm{d}\epsthet\,\mathrm{d}\theta  /
\pulse(\theta_0\!+\!\frac{\epsthet}{\varepsilon}+\theta,\widehat{x}_0(\epsthet))$.
Finally, performing the change of variable $\lambda=\ell-\varepsilon \mu(\theta)$, with
$\mu(\theta)$ defined in \S\ref{sec-not-S1}, yields
\begin{equation*}
  \tau=
\int_{\varepsilon\mu(\theta)}^{\mathscr{P}(\tKep)+\varepsilon\mu(\theta)}
\left(
\frac1{2\pi}
\int_{0}^{2\pi}
\frac{\mathrm{d}\theta}{\pulse(\theta_0\!+\!\frac{\ell}{\varepsilon},\widehat{x}_0(\ell-\varepsilon\mu(\theta)))}\right) \mathrm{d}\ell
\end{equation*}
Using \eqref{eq:14}, the fact that $|\mu(\theta)|<2\pi$ and Lipschitz continuity of both $\widehat{x}_0$ and $\pulse$,
we deduce from this and equation \eqref{eq:6} that 
$\left|\mathfrak{T}\!\circ\!\mathscr{P}(\tKep)-\frac\tau\varepsilon\right|\leq|\rho|+k'$ for some constant $k'$ and
finally, using the fact that $\III_\varepsilon(.)$ is Lipschitz continuous with constant $2\,\varepsilon\sup\|G\|/k_\pulse$, one has
$d(\III_\varepsilon(\mathfrak{T}\circ\mathscr{P}(\tKep)),\III_\varepsilon(\frac\tKep\varepsilon))<c'''\varepsilon$ for
some constant $c'''$. This and equation \eqref{eq:27} imply
implies point 1 of the theorem, with $c=c'\!+\!c''\!+\!c'''$ in \eqref{eq:CVkep}.

For point 2, consider $\bigl(\theta_n(.),\III_{n}(.)\bigr):\,[0,T/\varepsilon_n]\to S^1\times\Set{K}$ 
a solution of system~(\ref{eq:Kep}) with $\varepsilon=\varepsilon_n$ and some control
$u=u_n(t)$. 
Following \eqref{eq:13}--\eqref{eq:001} and setting $\epsthet=\mathscr{R}_n(t)$ (we write
$\mathscr{R}_n$ because $\mathscr{R} $ in \eqref{eq:13} is constructed for system
$(\widetilde{\Sigma}_{\theta_0,\varepsilon_n})$ and thus depends on $n$),
one associates to these $x$ and $u$ a control $\epsthet\mapsto \widetilde{u}_n(\epsthet)$ and a solution $\epsthet\mapsto \widetilde{\III}_{n}(\epsthet)$ of $(\widetilde{\Sigma}_{\theta_0,\varepsilon_n})$. The solutions
$\epsthet\mapsto \widehat{\III}_{n}(\epsthet)$ of $(\widehat{\Sigma}_{\theta_0,\varepsilon_n})$ with same control and
initial condition satisfy, for the same reasons as \eqref{eq:26},
$d(\widehat{\III}_{n}(\epsthet),\widetilde{\III}_{n}(\epsthet))<c''\varepsilon_n$ for some constant $c''$. By Theorem~\ref{th-conv-sysM}
(Point 2), the sequence $(\widehat{x}_n)$ is compact and subsequences converge to solutions $\epsthet\mapsto\widehat{x}_0(\epsthet)$
of \eqref{eq:24},
hence the same subsequences of $(\widetilde{x}_n)$ converge as well, and, with
$  \tKep=\mathscr{Q}(\epsthet)\stackrel\Delta=\int_0^\epsthet
\frac{\mathrm{d}\ell}{\overline{\pulse}\left(\widehat{x}(\ell)\right)}$,
the maps $\tau\mapsto \widetilde{x}_n(\mathscr{Q}^{-1}(\tKep))=x_n((\mathscr{Q}\circ\mathscr{R}_n)^{-1}(\tKep))$
converge to a solution $\tKep\mapsto \III_0(\tKep)=\widehat{x}_0(\mathscr{Q}^{-1}(\tKep))$ of the average system
\eqref{eq:21}, with distance less than $c'\varepsilon_n$ for some constant $c'$.  
Using the same argument as in Point~1 for $\mathfrak{T}\!\circ\!\mathscr{P}(\tKep)$, one gets a bound for 
$|(\mathscr{Q}\!\circ\!\mathscr{R}_n)^{-1}(\tKep)-\frac\tKep {\varepsilon_n}|$ and, for some constant $c'''$, \\
$d(\,x_n((\mathscr{Q}\circ\mathscr{R}_n)^{-1}(\tKep))\,,\,x_n(\frac\tKep\varepsilon))\leq c'''\varepsilon_n$. Point 2 is proved.
\qquad\end{proof}

\subsection{Dimension of $\EE(x)$}
In \S\ref{sec-propav}, and in particular in Proposition~\ref{lem-dim}, $\GG$ can simply be replaced with $G$. It is however interesting to give a more intrinsic
characterization of $r(\theta,x)$ and thus of $\dim\EE(x)$.
\begin{proposition}
  \label{prop-intrinsic-Kep} If \eqref{eq:Kep} and \eqref{eq:Kep-champs} represent the same control system, then
  \begin{align}
    \nonumber
    \dim\Bigl(\sum_{j\in\NN}&\Range\frac{\partial^j G}{\partial \theta^j}(\theta,x)\Bigr)
\\[-1ex]
    \label{eq:2}
&=-1+\rank\Bigl(\{f_0(\theta,x)\}\cup\left\{\mathrm{ad}_{f_0}^j f_k\,(\theta,x)\;,\ j\in\NN, 1\leq k\leq
    m\right\}\Bigr)  \;.
  \end{align}
\end{proposition}

\vspace{-1.3\baselineskip}

\begin{proof}
  Straightforward computation using the fact that $f_0=\partial/\partial\theta$.
\hfill\end{proof}

\smallskip

Note that the right-hand side is $r(\theta,x)$. Proposition \ref{lem-dim} applies, with this definition of $r$. In
particular, the ``full rank case'' becomes:
\begin{proposition}
  \label{prop-open-Kep} If the vector fields $f_0$ and $\mathrm{ad}_{f_0}^j f_k$, $1\leq k\leq  m$, $j\in\NN$
  span the whole tangent space of $S^1\times\vari$, then $\EE(x)$ has a nonempty interior for all $x$.
\end{proposition}

\subsection{The function $H(x,p)$}
Instead of \eqref{eq:Hamiltonian}, $H$ has to be taken as follows, with $\overline{\pulse}$ defined in \eqref{eq:20}:
\begin{equation}
  \label{eq:HamKep}
  H(x,p) =\ 
  \overline{\pulse}(x)\;
  \frac{1}{2\pi}\!\!
  \int_0^{2\pi}\bigl\|\left\langle p,\frac{G(\theta,x)}{\pulse(\theta,x)}\right\rangle\bigr\|\mathrm{d}\theta
\;.
\end{equation}
The characterization of $\EE(x)$ in Proposition~\ref{lem-E-H} is unchanged.
In the ``full rank case'', the results from \S\ref{sec-fullrank} on the degree of differentiability apply without a change. 

\subsection{Application to the minimum time problem} As in \S\ref{sec-Tmin}, but for the Kepler system \eqref{eq:Kep},
let $x_0$, $x_1$ be fixed, call $T_\varepsilon$ the minimum time such that, from \emph{some} $\theta_0,\theta_1$,
$(\theta_1,x_1)$ can be reached from $(\theta_0,x_0)$ in system $(\mathcal{K}_\varepsilon)$ (obviously
$T_\varepsilon\to+\infty$ as $\varepsilon\to 0$) and $T_0$ the minimum time such that $x_1$ can be reached
from $x_0$ in the average system \eqref{eq:21}.
The equivalent of Theorem~\ref{prop-limT}, with a similar proof, using Theorem~\ref{th-conv-sysM-Kep}, is:

\smallskip

\begin{theorem}
\label{prop-limTKep}
In the full rank case, one has $\ \limsup_{\varepsilon\to0} \varepsilon T_\varepsilon\leq T_0\ $
(hence $\varepsilon\, T_\varepsilon$ is bounded as $\varepsilon\!\to\!0$).
If, for all $\varepsilon>0$ small enough, there is a minimizing solution
$(\theta_\varepsilon,x_\varepsilon):[0,T_\varepsilon]\to S^1\times\vari$ and they all remain in a common compact subset of
$\vari$, then 
all accumulation points (as $\varepsilon\!\to\!0$) of the compact family $(\tKep\to
x_\varepsilon(\frac\tKep\varepsilon))_{\varepsilon>0}$ in $C^0([0,T_0],\vari)$ are minimizing for the average system and
$\ \lim_{\varepsilon\to0}\varepsilon\,T_\varepsilon=T_0\ .$
\end{theorem}

\smallskip

The Hamiltonian for minimum time for the average system is given by \eqref{eq:HamKep}; one has to perform the time
scaling described in \S\ref{sec-kep-trsf} to have a result like Theorem~\ref{th-ham-flow} and the simple ``commutation
between averaging and writing Hamiltonian'' noted in Remark~\ref{rmk-HamCom}. Let us translate in terms of
\eqref{eq:Kep} the sufficient condition for existence of a Hamiltonian flow given by Theorem~\ref{th-ham-flow}:
\begin{theorem} 
\label{th-ham-flow-K}
In the full rank case, 
assume that $\langle p , G(\theta,x) \rangle$ and $\langle p ,
    {\partial G}/{\partial\theta}(\theta,x) \rangle$ do not vanish simultaneously outside $\{p=0\}$, that 
    $\theta\mapsto\langle p, G(\theta,x) \rangle$ vanishes at most once for each $(x,p)\!\in\mathrm{T}^*\vari$, $p\neq0$, and that
    $\rank\GG(\theta, x)=m$ for each $(\theta,x)\in S^1\times\vari$.
Then \eqref{eq-sys-maxHmoy}, with $H$ given by \eqref{eq:HamKep}, has a unique solution for any initial condition.
\end{theorem}

The discussion that follows Theorem~\ref{th-ham-flow} also applies to the above; let us mention that, once it has been checked
that, for each $(x,p)$, $\langle p, G(\theta,x) \rangle$ vanishes for at most one $\theta$, the other conditions are
guaranteed if \eqref{eq:rank-2} holds with $\GG$ replaced by $G$ or, in terms of the vector fields in \eqref{eq:Kep-champs}, if, for all $\xi=(\theta,x)$,
\begin{equation}
  \label{eq:7}
  \rank\{f_0(\xi) ,f_1(\xi),\ldots,f_m(\xi) ,\mathrm{ad}_{f_0}f_1(\xi),\ldots,\mathrm{ad}_{f_0}f_m(\xi)\}=n+1.
\end{equation}

We prove in the next section that the above conditions are true for the planar control 2-body problem.

\section{Application to the controlled 2-body system}
\label{sec-2body}

In this section we study some properties of the planar control system and demonstrate that it satisfies the condition of Theorem~\ref{th-ham-flow} on the domain of non-degenerated elliptic orbits. 

\subsection{Planar control 2-body system}
It is classically described by some first integrals of the free movement ---here the semi-major
axis $a$ and the eccentricity vector $(e_x,e_y)$--- and one angle $\LLL$ following the dynamics;
we restrict to the set of non-degenerated elliptic orbits rotating in the direct sense, i.e. the state space is
$S^1\times\vari$ with $\vari=\{(a,e_x,e_y)\in\RR^3,\;a>0\,\text{ and }\,{e_x}^2+{e_y}^2<1\}$.
The control $u=(u_t,u_n)$ is expressed in the tangential-normal frame and the system reads:
\begin{eqnarray}
  \label{eq:motion}
  \frac{\mathrm{d}}{\mathrm{d}t}\!
\left(\!\!
\begin{array}{c}
a\\
e_x\\
e_y \\
\LLL
\end{array}
\!\!\right)
&=&
\frac1{a^{3/2}}\!\!\left(\!\!\!
\begin{array}{c}
0\\0\\0\\\mathtt{w}(e_x,e_y,\LLL)
\end{array}
\!\!\!\right)
+
\sqrt a \,
\left(\!\!
  \begin{array}{cc}
2\,a\,\mathtt{a}_a(e_x,e_y,\LLL) & 0
\\
2\,\mathtt{a}_x(e_x,e_y,\LLL) & \!\!\mathtt{b}_x(e_x,e_y,\LLL)
\\
2\,\mathtt{a}_y(e_x,e_y,\LLL) & \!\!\mathtt{b}_y(e_x,e_y,\LLL)
\\0&0
  \end{array}
\!\!\!
\right)
\!\!\!\!
  \begin{array}{c}\left(\!\!\!
  \begin{array}{c}
u_t\\u_n
  \end{array}
\!\!\!\right)
\\
\ \\\ 
\end{array}
\end{eqnarray}
\begin{align*}
  \text{with}\ \ \
\mathtt{w}(e_x,e_y,L) &=
\frac 
{ \left( 1+e_x \cos \LLL+ e_y \sin \LLL  \right) ^{2}}
{ \left( 1-{e}^{2} \right) ^{3/2}}\,,
\displaybreak[0]
\\
\mathtt{a}_a(e_x,e_y,\LLL) &=
\frac 
{\sqrt {1+{e}^{2}+2\,e_x \cos \LLL+2\, e_y \sin \LLL }}
{\sqrt {1-{e}^{2}}}\,,
\displaybreak[0]
\\
  \mathtt{a}_x(e_x,e_y,\LLL)&=
\frac 
{\sqrt {1-{e}^{2}}}
{\sqrt {1+{e}^{2}+2\,e_x \cos \LLL+2\, e_y \sin \LLL }}
(e_x + \cos \LLL)\,,
\displaybreak[0]
\\
  \mathtt{a}_y(e_x,e_y,\LLL)&=
\frac 
{\sqrt {1-{e}^{2}}}
{\sqrt {1+{e}^{2}+2\,e_x \cos \LLL+2\, e_y \sin \LLL }}
(e_x + \cos \LLL)\,,
\displaybreak[0]
\\
  \mathtt{b}_x(e_x,e_y,\LLL) &=
\frac 
{\sqrt {1-{e}^{2}}}
{\sqrt {1+{e}^{2}+2\,e_x \cos \LLL+2\, e_y \sin \LLL }}
\\&\hspace{6em}
\times \frac
{-2\,e_y+(e_x^2-e_y^2-1)\sin \LLL - 2\, e_x e_y \cos \LLL}
{1+2\,e_x \cos \LLL+2\, e_y \sin \LLL}\,,
\displaybreak[0]
\\
 \mathtt{b}_y(e_x,e_y,\LLL) &=
\frac 
{\sqrt {1-{e}^{2}}}
{\sqrt {1+{e}^{2}+2\,e_x \cos \LLL+2\, e_y \sin \LLL }}
\\&\hspace{6em}
\times \frac
{2\,e_x+(e_x^2-e_y^2+1)\cos \LLL + 2\, e_x e_y \sin \LLL}
{1+2\,e_x \cos \LLL+2\, e_y \sin \LLL}\,.
\end{align*}
The eccentricity $e$ is the norm of the eccentricity vector, $e=\sqrt{e_x^2+e_y^2}$. Low thrust translates into $\|u\|
\leq \varepsilon$ for a small $\varepsilon$. 


\refstepcounter{theorem}\paragraph{Remark~\thetheorem}\label{rmk-2body}
This is indeed a ``Kepler control system'' of the type \eqref{eq:Kep} except that $\pulse=\mathtt{w}/a^{3/2}$ is, although
strictly positive, not bounded from below by a positive constant on $S^1\times\vari$. There is such a lowerbound if one replaces
$\vari$ by $\vari^{\bar{c}}=\{(a,e_x,e_y)\in\RR^3,\;a>0\,\text{ and}\,{e_x}^2+{e_y}^2<\bar{c}\}$ with $\bar{c}<1$.
Strictly speaking, the results of the paper have to be applied in $\vari^{\bar{c}}$, $\bar{c}<1$.
However, Theorems~\ref{th-conv-sysM-Kep} or \ref{th-ham-flow-K}, for instance, may be applied in $\vari$ because each statement may ultimately
be restricted to a compact subset of $\vari$, itself included in some $\vari^{\bar{c}}$, $\bar{c}<1$.

\smallskip

The Hamiltonian that both defines the average system according to \eqref{eq:12} and yields the Hamiltonian system governing
extremals for minimum time is given by
\eqref{eq:HamKep}. Since $\int_0^{2\pi}\mathrm{d}\LLL/\mathtt{w}(e_x,e_y,\LLL)=2\pi$,
it can be expressed as\\ $H(a,e_x,e_y,p_a,p_{e_x},p_{e_y})=\sqrt{a} \mathcal{H}(e_x,e_y,a p_a,p_{e_x},p_{e_y})$ with
\begin{align*}
\mathcal{H}(e_x,e_y,A,X,Y)=&\frac{1}{2\pi} \int_0^{2\pi}\| (A\, X\, Y) \mathtt{G}(e_x,e_y,\LLL) \|, 
\\
\mathtt{G}(e_x,e_y,\LLL) =&
\left(
  \begin{array}{cc}
2\,\mathtt{a}_a/\mathtt{w} & 0
\\
2\,\mathtt{a}_x/\mathtt{w} & \mathtt{b}_x/\mathtt{w}
\\
2\,\mathtt{a}_y/\mathtt{w} & \mathtt{b}_y/\mathtt{w}
 \end{array}
\right).
\end{align*}

\subsection{Hamiltonian flow} Theorem~\ref{th-ham-flow-K} applies to this system. Indeed:

\begin{proposition}
\label{prop-zero-2body}
Fore each $(e_x,e_y,a)$ with ${e_x}^2+{e_y}^2<1$ and $a>0$, and each $(A,X,Y)\!\neq\!(0,0,0)$, the vector
$(A\, X\, Y) \mathtt{G}(e_x,e_y,\LLL)$ vanishes for at most one angle $\LLL$.
\end{proposition}

\begin{proof}
Removing denominators, the equations $A\mathtt{a}_a+X\mathtt{a}_x+Y\mathtt{a}_y=0$ and $X\mathtt{b}_x+Y\mathtt{b}_y=0$
can be written:
\begin{align*}
    \left(2\, e_x A + 2 (1-e^2) X \right) \cos \LLL 
    + \left(2\, e_y  A \vphantom{e^2}\right.&\left.+ 2(1-e^2) Y \right) \sin \LLL
    \\ &= - (1+e^2) A - 2\, e_x (1-e^2) X - 2\,e_y (1-e^2) Y
  \\
    \left(-2\, e_x e_y X + (e_x^2-e_y^2+1) Y \right) \cos \LLL
    +&\left( (e_x^2-e_y^2-1) X + 2\,e_x e_y Y \right) \sin \LLL
    \\ &= 2\,e_y X - 2\, e_x Y .
\end{align*}

If $
\Delta = 
\left|
\begin{array}{cc}
2\, e_x A + 2 (1-e^2) X & 2\, e_y  A + 2(1-e^2)Y \\
-2\, e_x e_y X + (e_x^2-e_y^2+1) Y  & (e_x^2-e_y^2-1) X + 2\,e_x e_y Y 
\end{array}
\right|
\vphantom{\displaystyle\frac1{\displaystyle \frac12}}
$ is nonzero, there is clearly at most one solution $\LLL$.
If $\Delta = 0$, there exists $\lambda \neq 0$ such that 
\begin{eqnarray*}
2\,e_x A +2 (1-e^2) X & = & \lambda \left( -2\, e_x e_y X + (e_x^2-e_y^2+1) Y \right), \\
2\,e_y A +2 (1-e^2) Y & = & \lambda \left( (e_x^2-e_y^2-1) X + 2\, e_x e_y Y \right),
\end{eqnarray*}
and there may be a solution to the system above only if
\[
(1+e^2) A + 2\,e_x(1-e^2) X + 2\, e_y (1-e^2) Y = - 2\, \lambda (e_y X-e_x Y) 
\]
These three equations forms a linear system in $(A,X,Y)$, $M (A,X,Y)^T = 0$ with  
\[
M=
\left(
\begin{array}{ccc}
2 e_x & 2(1-e^2+\lambda e_x e_y) & -\lambda (e_x^2-e_y^2+1) \\
2 e_y & - \lambda (e_x^2-e_y^2-1) & 2(1-e^2-\lambda e_x e_y) \\
(1+e^2) & 2 \left( e_x (1-e^2) + \lambda e_y \right) & 2 \left( e_y (1-e^2) - \lambda e_x \right)
\end{array}
\right).
\]
A brief computation gives $\det M=(1-e)^3(1+e)^3 (\lambda^2 + 4)$, strictly positive when $0\leq e<1$. Hence $M (A,X,Y)^T = 0$ implies $(A,X,Y)=0$. 
\hfill\end{proof}

Since the rank of $\mathtt{G}$ is obviously equal to 2 and the rank of $\{\mathtt{G},\partial \mathtt{G}/\partial\LLL\}$
equal to 3 for
any $(e_x,e_y,\LLL)$, the hypotheses of Theorem~\ref{th-ham-flow-K} are satisfied by the planar control 2-body system,
and it guarantees
existence of a flow for the Hamiltonian system governing the extremals of minimum time for its average system.

\section{Conclusion}
\label{sec-concl}
Attempting to formulate a control theory equivalent to the averaging theorems for ODEs
naturally leads to, and justifies, the notion of average control system introduced in this paper.
It has a conceptual importance as well as, for instance, applications to approximation of minimum time control.

Besides its definition and description, we gave results on its regularity and on the dimension of its velocity set
(``number of inputs''). These are however mostly a starting point. The regularity of $H$ has to be further
explored when the conditions of Theorem~\ref{th-ham-flow} do not hold, see the last paragraph of \S\ref{sec-fast}.

It has already allowed us to give (with restrictions on the
eccentricities, see Remark~\ref{rmk-2body}) a proof~\cite{bom-pom:07:DCDS} that the minimum time between 2 ellipses
grows like $1/\varepsilon$ for the planar 2-body problem.
Here also, progress must be made.
Explicit computation of the average system and its extremals for the 2-body problem has to be conducted.

\section*{\bf Acknowledgements} The authors are indebted to Jana N{\v e}mcov{\'a} from Institute of Chemical
Technology, Prague, for a careful proof-reading of the draft manuscript, and to two anonymous referees from this
journal for extremely constructive reviews that make this paper considerably easier to read than the original
submission.

\bigskip

\Appendix
\section{Proof of Proposition~\ref{prop-C1-LipLog}}

{\em Proof of Point \ref{point-C1}.}
The integral in \eqref{eq:H'} is well defined (its integrand is bounded) and, by \eqref{eq:8} and Lebesgue
convergence theorem, it is continuous with respect to $X$ and $h$. Let us prove that this 
$\mathrm{d}\mathsf{H}$ \emph{is} the derivative of $\mathsf{H}$.
  Since $\mathsf{V}$ is smooth, one has\vspace{-1ex}
  \begin{equation}
    \label{eq:00}
    \|\mathsf{V}(\theta,X+h)-\mathsf{V}(\theta,X)-\frac{\partial\mathsf{V}}{\partial X}(\theta,X).h\|\leq k\,\|h\|^2\,,
\vspace{-1ex}
  \end{equation}
  where $\frac{\partial\mathsf{V}}{\partial X}(\theta,X)$ is smooth with respect to $(\theta,X)$ and $k$ is some local
  constant. Now, assuming $\mathsf{V}(\theta,X)\neq0$, one has 
\vspace{-1.2ex}
  \begin{eqnarray*}
    \|\mathsf{V}(\theta,X+h)\|-\|\mathsf{V}(\theta,X)\|&=&
    \scal{\mathsf{V}(\theta,X+h)-\mathsf{V}(\theta,X)}{\frac{\mathsf{V}(\theta,X)}{\|\mathsf{V}(\theta,X)\|}}
    \\ &&
    \hspace{4.5em}+\;a(\theta,X,h)\,\frac{\|\mathsf{V}(\theta,X+h)-\mathsf{V}(\theta,X)\|^2}{\|\mathsf{V}(\theta,X)\|+\|\mathsf{V}(\theta,X+h)\|}\,
  \end{eqnarray*}
  with $|a(\theta,X,h)|\leq2$. 
Hence, from (\ref{eq:00}) and (\ref{eq:H'}), one has, for some local constant $k'$,
$$\frac{\|\mathsf{H}(X+h)-\mathsf{H}(X)-\mathrm{d}\mathsf{H}(X).h\|} {\|h\|}
\leq \frac{k'}{2\pi} \int_0^{2\pi} \left(\|h\|+\frac{\|\mathsf{V}(\theta,X+h)-\mathsf{V}(\theta,X)\|}{\|\mathsf{V}(\theta,X)\|+\|\mathsf{V}(\theta,X+h)\|}
\right)\mathrm{d}\theta
$$
for $\|h\|$ small enough.
For fixed $X$ and $h\to0$, the integrand in the right-hand side is bounded by $1+\|h\|$ and converges to zero for
$\theta$ outside the set $\{\theta\in S^1,\;\mathsf{V}(\theta,X)=0\}$: by (\ref{eq:8}) and Lebesgue convergence theorem,
the right-hand side tends to zero.
\qquad\endproof

\smallskip
Let us now state two lemmas that are needed in the proof of Point \ref{point-LipLog}.
\begin{lemma}
  \label{lem-chihat}
  Assume that $\bar{X}\in\widetilde{\mathcal{Z}}$ and \eqref{eq:A123} is satisfied. There is a neighborhood $\ouv$ of
  $\bar{X}$ in $\varid$ and a smooth map $\widehat{\chi}:\ouv\to S^1$ such that, for $(\theta,X)\in\ouv$, one has
  $\mathsf{V}(\theta,X)=0$ only if $\theta=\widehat{\chi}(X)$, and
\vspace{-.2\baselineskip}
\begin{equation}
  \label{eq:53}
  \scal {\frac{\partial \mathsf{V}}{\partial\theta}(\widehat{\chi}(X),X)} {\mathsf{V}(\widehat{\chi}(X),X)}
=0\,,\ \ X\in\ouv\,.
\vspace{-.4\baselineskip}
\end{equation}
\end{lemma}
\begin{proof}
  From (\ref{eq:A123}.a), $\mathcal{Z}=\{(\theta,X)\in S^1\times\varid\,,\;\mathsf{V}(\theta,X)=0\}$ is a smooth submanifold of
  $S^1\times\varid$ and from (\ref{eq:A123}.c), $\widetilde{\mathcal{Z}}$ given by \eqref{eq:H-Ztilde} a smooth submanifold of
  $\varid$, both of dimension $d+1-m$, and the projection $\pi:S^1\times\varid\to\varid$ induces a diffeomorphism
  $\mathcal{Z}\to\widetilde{\mathcal{Z}}$ whose inverse is of the form $x\mapsto(\chi(x),x)$ with $\chi$ a smooth map
  $\widetilde{\mathcal{Z}}\to S^1$ that satisfies, for all $X\in\widetilde{\mathcal{Z}}$: $\mathsf{V}(\theta,X)=0$ if an only if $\theta=\chi(x)$.

  Consider the map $T:S^1\times\varid\to\RR$ given by
  $T(\theta,X)=\scal {\frac{\partial \mathsf{V}}{\partial\theta}(\theta,X)} {\mathsf{V}(\theta,X)}$.
  Let $\bar{X}$ be in $\widetilde{\mathcal{Z}}$; since $\mathsf{V}(\chi(\bar{X}),\bar{X})=0$, one has
  $T(\chi(\bar{X}),\bar{X})=0$ and $\partial T/\partial\theta(\chi(\bar{X}),\bar{X})= 
  \|\frac{\partial \mathsf{V}}{\partial\theta}(\chi(\bar{X}),\bar{X})\|^2$, nonzero from assumption (\ref{eq:A123}.b): the implicit function
  theorem yields a unique map $\widehat{\chi}$ from a neighborhood $U$ of $\bar X$ in $\varid$ to a neighborhood of
  $\chi(\bar X)$ in $S^1$ such that $\theta=\widehat{\chi}(X)$ solves $T(\theta,X)=0$; it must
  therefore coincide with $\chi$ in $U\cap\widetilde{\mathcal{Z}}$ and satisfies the lemma. 
\qquad\end{proof}
\begin{lemma}
  \label{lem-coord}
Assume that $\bar{X}\in\widetilde{\mathcal{Z}}$ and \eqref{eq:A123} is
satisfied. There exist a neighborhood $\ouv$ of $\bar{X}$ in $\varid$, local coordinates
$x_1,\ldots,x_d$ defined on $\ouv$, and
smooth maps \\$P\!:\ouv\to SO(m)$, $\alpha\!:\ouv\to\RR$, and 
$W\!:S^1\times\ouv\to\RR^m$ such that, with $X_{\mathbf{I}}\!=\!{\scriptstyle\left(\!\!\!
    \begin{array}{c}
x_1\\[-.7ex]\vdots\\[-1ex] x_{m-1}
    \end{array}
\!\!\!\right)}$,  
\vspace{-1.2ex}
\begin{eqnarray}
  \label{eq:Vdev}
  \mathsf{V}(\theta,X)&=&P(X)\Bigl[\left(
    \begin{array}{c}
X_{\mathbf{I}}\\\alpha(X)\,\left(\theta-\widehat{\chi}(X)\right)
    \end{array}
\right)
+\left(\theta-\widehat{\chi}(X)\right)^2\,W(\theta,X)
\Bigr]
\\
  \label{eq:Vdev2}
 &=&P(X)\Bigl[
{\scriptstyle\left(\!\!\!
    \begin{array}{c}
X_{\mathbf{I}}\\0
    \end{array}
\!\!\!\right)}
+\left(\theta-\widehat{\chi}(X)\right)\,W_1(\theta,X)
\Bigr]
\\[-0.5ex]
\label{eq:W1}
&&\text{with}\ \; W_1(\theta,X)=
{\scriptstyle\left(\!\!\!
    \begin{array}{c}
0_{m-1}\\\alpha(X)
    \end{array}
\!\!\!\right)}
+\left(\theta-\widehat{\chi}(X)\right)\,W(\theta,X) \ ,
\end{eqnarray}
in $S^1\times\ouv$, where $\alpha$ is bounded from below: $0<\alpha_0<\alpha(X)$, $X\in\ouv$. 
Furthermore, for a constant $K_3>0$, one has, for all $(\theta,X)\in S^1\times\ouv$,
\begin{align}
  \label{eq:10}
  \|\mathsf{V}(\theta,X)\|\geq& \;{K_3}\sqrt{\|X_{\mathbf{I}}\|^2+\alpha(X)^2\left(\theta-\widehat{\chi}(X)\right)^2}\,,
\ \ \ \ \ \ \ \ \ \ \ \
\\
  \label{eq:11}
\text{and}\ \ \ \ \ \ \ \ \ \ \ \ \ \ \ \ \ \  X_{\mathbf{I}}=0\ \Rightarrow&\ 
  \|W_1(\theta,X)\|\geq K_3\,.
\end{align}
\end{lemma}
\begin{proof}
The map $X\mapsto\frac{\partial{\mathsf{V}}}{\partial\theta}(\widehat{\chi}(X),X)$ is nonzero for $X=\bar X$, hence it does not vanish on
a sufficiently small neighborhood $\ouv$ of $\bar X$, and one may write
\begin{equation}
\label{eq:54}
  \frac{\partial {\mathsf{V}}}{\partial\theta}(\widehat{\chi}(X),X)\ =\ P(X)\;\left(
    \begin{array}{c}
0_{m-1}\\\alpha(X)
    \end{array}
\right)\ ,\ \ \alpha(X)>\alpha_0>0\ .
\end{equation}
Define $v_1,\ldots,v_m$, smooth maps  $S^1\times\ouv\to\RR$ by 
\begin{equation}
  \label{eq:56}
  \left(
    \begin{array}{c}
v_1(\theta,X)\\\vdots\\v_m(\theta,X)
    \end{array}
\right)
=  P^{-1}(X)\;\mathsf{V}(\theta,X)\,.
\end{equation}
For $i$ between 1 and $m-1$, $\frac{\partial v_i}{\partial\theta}(\widehat{\chi}(X),X)=0$ from (\ref{eq:54}), and 
$v_i(\widehat{\chi}(\bar X),\bar X)=0$ from Lemma~\ref{lem-chihat} and, using(\ref{eq:A123}.a), the rank of the map
$X\mapsto(v_1(\widehat{\chi}(X),X),\ldots,$ $v_{m-1}(\widehat{\chi}(X),X))$ is $m-1$ at $X=\bar X$: on a possibly smaller
neighborhood $\ouv$, there are local coordinates $x_1,\ldots,x_{d}$ such that
$v_i(\theta,X)=x_i+\left(\theta-\widehat{\chi}(X)\right)^2\,W_i(\theta,X)$ for $i\leq m-1$ and for some smooth $W_i$; 
substituting (\ref{eq:54}) and (\ref{eq:56}) in (\ref{eq:53}) implies $v_m(\widehat{\chi}(X),X)=0$, hence
$v_m(\theta,X)=\alpha(X)\,\left(\theta-\widehat{\chi}(X)\right)+W_m(\theta,X) \left(\theta-\widehat{\chi}(X)\right)^2$
for a smooth $W_m$; \eqref{eq:Vdev} is proved.

Possibly restricting $\ouv$ to a subset with compact closure, $\|W(\theta,X)\|$ is bounded on $S^1\times\ouv$;
if $|\theta-\widehat{\chi}(X)|\leq\frac12\alpha_0/\max \|W\|$, then \eqref{eq:10} holds with
$K_3=\frac12$ according to \eqref{eq:Vdev}; on the set where $|\theta-\widehat{\chi}(X)|\geq\frac12\alpha_0/\max
\|W\|$, $\mathsf{V}$ does not vanish and hence 
$(\|X_{\mathbf{I}}\|^2+\alpha(X)^2\left(\theta-\widehat{\chi}(X)\right)^2)^{1/2}/\|\mathsf{V}(\theta,X)\|$
is bounded from below; \eqref{eq:10} is proved, with $K_3$ smaller than this bound and than $\frac12$. 
From \eqref{eq:W1}, $W_1(\widehat{\chi}(\bar X),\bar X)\neq0$ because $\alpha$ does not vanish; from assumption (\ref{eq:A123}.b) and
\eqref{eq:Vdev2} (where $X_{\mathbf{I}}=0$ if $X=\bar X$), $W_1(\theta,\bar X)\neq0$ if
$\theta\neq\widehat{\chi}(\bar X)$, hence $W_1$ does not vanish on $S^1\times\{\bar X\}$; it is therefore bounded from
below on $S^1\times\ouv$ with $\ouv$ a small enough neighborhood of $\bar X$: \eqref{eq:11} holds with $K_3$ smaller
than this bound.
\hfill\end{proof}

\smallskip

{\em Proof of Proposition \ref{prop-C1-LipLog} (Point \ref{point-LipLog})}. \emph{We use $[-\pi,\pi]$ instead of $[0,2\pi]$ as an interval of integration}.
Let $h\in\RR^{\mathsf{d}}$, with $\|h\|=1$. From (\ref{eq:H'}), one has, for some constant $\widetilde{K}$  using bounds on the derivatives of the smooth $\mathsf{V}$,
\begin{align*}  
\left|\mathrm{d}\mathsf{H}(X).h-\mathrm{d}\mathsf{H}(Y).h \right| \leq&
\left|
\frac1{2\pi}\int_{-\pi}^{\pi}\scal{
\frac{\partial\mathsf{V}}{\partial X}(\theta,X).h-\frac{\partial\mathsf{V}}{\partial X}(\theta,Y).h
}
{\frac{\mathsf{V}(\theta,X)}{\|\mathsf{V}(\theta,X)\|}}
\mathrm{d}\theta   \right|
\\
&\ +\left|
\frac1{2\pi}\int_{-\pi}^{\pi}\scal {\frac{\partial\mathsf{V}}{\partial X}(\theta,Y).h} {
\frac{\mathsf{V}(\theta,X)}{\|\mathsf{V}(\theta,X)\|}-\frac{\mathsf{V}(\theta,Y)}{\|\mathsf{V}(\theta,Y)\|}}
\mathrm{d}\theta  \right|
\\
\leq\;
&\widetilde{K}\|X-Y\|+
\frac{\widetilde{K}}{2\pi}
\left\|
\int_{-\pi}^{\pi}\!\frac{\mathsf{V}(\theta,X)}{\|\mathsf{V}(\theta,X)\|}\mathrm{d}\theta
\,-\!
\int_{-\pi}^{\pi}\!\frac{\mathsf{V}(\theta,Y)}{\|\mathsf{V}(\theta,Y)\|}\mathrm{d}\theta
\right\|\,.
\end{align*}
Finally, defining
\begin{equation}
  \label{eq:4}
  \widehat{\mathsf{V}}(\varphi,X)=\mathsf{V}(\widehat{\chi}(X)+\varphi,X)\;,\ \ \ \widehat{W}_1(\varphi,X)=W_1(\widehat{\chi}(X)+\varphi,X)\;,
\end{equation}
and making a different change of variables in the last two integrals, one has
\begin{eqnarray}
  \nonumber
  \|\mathrm{d}\mathsf{H}(X).h-\mathrm{d}\mathsf{H}(Y).h \| &\leq& \widetilde{K}\|X-Y\| 
+
\frac{\widetilde{K}}{2\pi}
\int_{-\pi}^{\pi}
\bigl\|
\frac{\widehat{\mathsf{V}}(\varphi,X)}{\|\widehat{\mathsf{V}}(\varphi,X)\|}
-
\frac{\widehat{\mathsf{V}}(\varphi,Y)}{\|\widehat{\mathsf{V}}(\varphi,Y)\|} \bigr\|
\mathrm{d}\varphi
\\
\label{eq:9}
&\leq&
\widetilde{K}\|X-Y\|  +
\frac{\widetilde{K}}{\pi}
\int_{-\pi}^{\pi}
\frac{\|\widehat{\mathsf{V}}(\varphi,X)-\widehat{\mathsf{V}}(\varphi,Y)\|}
{\|\widehat{\mathsf{V}}(\varphi,X)\|}
\,\mathrm{d}\varphi
\end{eqnarray}
where the last inequality uses the fact 
$\|\frac{u}{\|u\|}-\frac{v}{\|v\|}\|\leq2\min\{\frac{\|u-v\|}{\|u\|},\frac{\|u-v\|}{\|v\|}\}$,
and also holds with $\|\widehat{\mathsf{V}}(\varphi,Y)\|$ instead of
$\|\widehat{\mathsf{V}}(\varphi,X)\|$ in the denominator.
Now let us use Lemma~\ref{lem-coord}, let $X=(x_1,\ldots,x_d)$ and $Y=(y_1,\ldots,y_d)$ in these coordinates; 
from \eqref{eq:Vdev2}, one has, with $\widehat W _1$ defined by \eqref{eq:4},
\vspace{-.8ex}
\begin{gather}
  \label{eq:3}
  \!\!\widehat{\mathsf{V}}(\varphi,X)= P(X)\Bigl[
{\scriptstyle \left(\hspace{-1.4ex}
    \begin{array}{c}
X_{\mathbf{I}}\\0
    \end{array}\hspace{-1.4ex}
\right)}
\!+
\varphi\,\widehat W_1(\varphi,X)
\Bigr]
\!,\;
\widehat{\mathsf{V}}(\varphi,Y)=P(Y)\Bigl[
{\scriptstyle \left(\hspace{-1.4ex}
    \begin{array}{c}
Y_{\mathbf{I}}\\0
    \end{array}\hspace{-1.4ex}
\right)}
\!+
\varphi\,\widehat W_1(\varphi,Y)
\Bigr]\!.
\\\nonumber
\text{Hence}\quad
\widehat{\mathsf{V}}(\varphi,X)-\widehat{\mathsf{V}}(\varphi,Y)=
\bigl(P(X)-P(Y)\bigr)P(X)^{-1} \widehat{\mathsf{V}}(\varphi,X)
\hspace{13.2em}\\[-0.3ex]
\nonumber
\,\hspace{13em}+
P(Y)\Bigl[\,\varphi
\Bigl(
W_1(\varphi,X)-W_1(\varphi,Y)
\Bigr)
+
{\scriptstyle \left(\hspace{-1.4ex}
    \begin{array}{c} X_{\mathbf{I}}-Y_{\mathbf{I}}\\0 \end{array}\hspace{-1.4ex}
\right)}
\,\Bigr]
\end{gather}
$\ $\\[-1.4\baselineskip]and finally
\begin{equation}
\label{eq:aa1} 
  \frac{\|\widehat{\mathsf{V}}(\varphi,X)-\widehat{\mathsf{V}}(\varphi,Y)\|}
{\|\widehat{\mathsf{V}}(\varphi,X)\|}
\leq
\|P(X)-P(Y)\|
+
\frac{|\varphi|\left\| W_1(\varphi,X)-W_1(\varphi,Y) \right\|}{\|\widehat{\mathsf{V}}(\varphi,X)\|}
+
\frac{\|X_{\mathbf{I}}-Y_{\mathbf{I}}\|}      
{\|\widehat{\mathsf{V}}(\varphi,X)\|}.
\end{equation}
Two cases are to be distinguished:
\begin{romannum}
  \item If $X_{\mathbf{I}}=Y_{\mathbf{I}}=0$, then $\varphi$ factors out of $\widehat{\mathsf{V}}(\varphi,X)$ and
    $\widehat{\mathsf{V}}(\varphi,Y)$ in \eqref{eq:3} and the last term in \eqref{eq:aa1} is zero: according to \eqref{eq:11},
   the integrand in \eqref{eq:9} is bounded by
$$\|P(X)-P(Y)\|+\frac{\|\widehat W_1(\varphi,X)-\widehat W_1(\varphi,Y)\|}{K_3}\;,$$
and finally $\left|\mathrm{d}\mathsf{H}(X).h-\mathrm{d}\mathsf{H}(Y).h \right|\leq K\,\|X-Y\| $ with a constant $K$ that depends only on $\mathsf{V}$, the open set $\ouv$
and the coordinates.
  \item  If $X_{\mathbf{I}}\neq0$ (or $Y_{\mathbf{I}}\neq0$, interchanging $X$ and
    $Y$), then \eqref{eq:aa1}, using \eqref{eq:10}, implies that the integrand in \eqref{eq:9} is bounded by
    \begin{displaymath}
      \|P(X)-P(Y)\|
+
\frac1{K_3}\;\frac1{\alpha_0}\;
\left\| W_1(\varphi,X)-W_1(\varphi,Y) \right\|
+
\frac1{K_3}\,\sqrt{ 
\frac{\|X_{\mathbf{I}}-Y_{\mathbf{I}}\|^2}{\|X_{\mathbf{I}}\|^2+\alpha(X)\varphi^2}}\ ,
    \end{displaymath}
but the same is also true replacing $\alpha(X)$ with $\alpha(Y)$ and $\|X_\mathbf{I}\|^2$ with
$\|Y_\mathbf{I}\|^2$; hence, since $\|a-b\|^2\leq4\max\{\|a\|^2,\|b\|^2\}$, the last term may be
replaced by
$\frac2{K_3}\,\sqrt{ 
\frac{\|X_{\mathbf{I}}-Y_{\mathbf{I}}\|^2}{\|X_{\mathbf{I}}-Y_{\mathbf{I}}\|^2+4\alpha_0\varphi^2}}
$,
whose integral between $-\pi$ and $\pi$ is equal to 
$$ \frac{\|X_{\mathbf{I}}-Y_{\mathbf{I}}\|}{K_3\,\sqrt{\alpha_0}}\,\ln(1+\frac{4\pi
  \sqrt{\alpha_0}}{\|X_{\mathbf{I}}-Y_{\mathbf{I}}\|}
+\frac{8\pi^2 \alpha_0}{\|X_{\mathbf{I}}-Y_{\mathbf{I}}\|^2}),$$
which is less
than $\|X_{\mathbf{I}}-Y_{\mathbf{I}}\|(k_1+k_2\ln\frac1{\|X_{\mathbf{I}}-Y_{\mathbf{I}}\|})$
for some $k_1,k_2$ when, say, $\frac{\|X_{\mathbf{I}}-Y_{\mathbf{I}}\|}{2\,\sqrt{\alpha_0}}<1$. Finally, 
since $\|X_{\mathbf{I}}-Y_{\mathbf{I}}\|$ is less than $\|X-Y\|$ and $u\mapsto u\ln(1/u)$ is nondecreasing, less than
$\|X-Y\|(k_1+k_2\,\ln\frac1{\|X-Y\|})$.
\end{romannum}

Cases (i) and (ii) do imply \eqref{eq:LipLog}, possibly restricting $\ouv$ so that $\ln\frac1{\|X-Y\|}\geq1$.
\hfill\endproof


\end{document}